\documentclass[12pt,a4paper]{amsart}
\usepackage[utf8]{inputenc}
\usepackage[T1]{fontenc}
\usepackage{ mathrsfs }
\usepackage{fixltx2e, graphicx, longtable, float, wrapfig, soul,
  textcomp, marvosym, wasysym, latexsym, hyperref, bbm,
  fancybox,fancyvrb, cite,amsmath,amssymb,amsthm,amsrefs,color,subfig,pdfpages,enumerate}
\usepackage[usenames,dvipsnames]{xcolor}
\usepackage{enumitem}
\usepackage[top=1.2in, bottom=1.2in, left=1in, right=1in]{geometry}
\usepackage{listings}
\usepackage{tabularx}
\usepackage{mathtools}
\usepackage[all]{xy}
\newcommand{\Aut}{\operatorname{Aut}}

\newcommand{\Reg}{\operatorname{Reg}}
\newtheorem{theorem}{Theorem}[section]
\newtheorem{lemma}[theorem]{Lemma}
\newtheorem{proposition}[theorem]{Proposition}
\newtheorem{corollary}[theorem]{Corollary}

\theoremstyle{definition}
\newtheorem{example}[theorem]{Example}

\newtheorem{definition}[theorem]{Definition}
\newtheorem{remark}[theorem]{Remark} 

\newcommand{\aut}{Aut} 

\newtheorem*{notation}{Notation}
\newtheorem{open}[theorem]{Open Problem} 

\numberwithin{equation}{section}

\newcommand{\Th}{\operatorname{Th}}

\newcommand{\N}{\mbox{$\mathbb{N}$}}

\newcommand{\Q}{\mbox{$\mathbb{Q}$}}

\renewcommand{\ar}{\mbox{{$\mathcal{R}$}}}
\newcommand{\el}{\mbox{{$\mathcal{L}$}}}

\newcommand{\dee}{\mbox{{$\mathcal{D}$}}}

\newcommand{\eh}{\mbox{{$\mathcal{H}$}}}
\newcommand{\jay}{\mbox{{$\mathcal{J}$}}}

\newcommand{\ale}{$\aleph_0$-categorical}

\newcommand{\X}{\mbox{$\mathcal{X}$}}

\newcommand{\Y}{\mbox{$\mathcal{Y}$}}

\title{$\aleph_0$-categoricity of semigroups}
\author{Victoria Gould and Thomas Quinn-Gregson}

\email{victoria.gould@york.ac.uk} 

\email{thomas.quinn-gregson@york.ac.uk}

\address{Department of Mathematics\\University
  of York\\York YO10 5DD\\UK}

\date{\today}



\thanks{This work forms part of the PhD of the second author at the University of York, funded by EPSRC}

\begin{document}

\begin{abstract}
  \noindent In this paper we initiate the study of $\aleph_0$-categorical semigroups, where a countable semigroup $S$ is $\aleph_0$-categorical if, for any natural number $n$, the action of  its group of automorphisms
  $\Aut S$ on $S^n$ has only finitely many orbits.  We show that $\aleph_0$-categoricity transfers to certain important substructures such as maximal subgroups and principal factors. 
  We  examine the relationship between  $\aleph_0$-categoricity and  a number of semigroup
  and monoid constructions, namely  direct sums, 0-direct unions,  semidirect products and $\mathcal{P}$-semigroups. As a corollary, we determine the $ \aleph_0$-categoricity of an $E$-unitary inverse semigroup with finite semilattice of idempotents in terms of that of the maximal group homomorphic image. 
\end{abstract}

\maketitle

\bibliographystyle{amsalpha}


%
%
%
%
%
%
%
%
%
%
%
%
%
%

\section{Introduction}  

 The concept of $\aleph_0$-categoricity is rooted in model theory. Let $A$ be a countable structure for a first-order language $L$, and let  $\Th(A)$  be  the set of all sentences of $L$ which are true in $A$. Then $A$ is {\em $\aleph_0$-categorical} if any other countable structure $B$ for $L$ with $\Th(A)=\Th(B)$ is such that $B$ is isomorphic to $A$.
 Thus, $A$ is $\aleph_0$-categorical if it is determined up to isomorphism by its first order theory. Morley's celebrated categoricity theorem \cite{Morley} was the impetus for the development of the rich area of model theory known as {\em stability theory}.

 It is a natural question to determine the $\aleph_0$-categorical members of any class of relational or algebraic structures: in our case, semigroups and monoids. This may be addressed without recourse to specialist model theory, in view of the following result,  independently accredited to Engeler \cite{Engeler}, Ryll-Nardzewski \cite{Ryll}, and Svenonius \cite{Svenonius}, but  commonly referred to as the Ryll-Nardzewski Theorem (RNT). Although stated in generality, we will apply it almost entirely in the context of semigroups, and semigroups with augmented structure (such as an identity, a partial order, or  distinguished subsets). 

\begin{theorem} \emph{(Ryll-Nardzewski Theorem)}
\label{RNT}
A countable structure $A$ is $\aleph_0$-categorical if and only if
$\Aut A$  has only finitely many orbits in its natural action on $A^n$ for each $n\in\N$. 
\end{theorem}

 A number of authors have considered $\aleph_0$-categoricity for algebraic structures.   Rosenstein  \cite{Rosenstein73}  classified   $\aleph_0$-categorical abelian groups, and an  extensive overview  of the results for groups is given in \cite{Apps82}. 
Baldwin and Rose \cite{BaldwinRose} investigated $\aleph_0$-categoricity for rings. 
For semigroups per se, little is known in this context. This paper and a sequel provide an introduction to the study of $\aleph_0$-categorical semigroups.

There are a number of directions in which  to study the property of $\aleph_0$-categoricity. The first half of this paper will be in line with the `preservation theorems' approach. In particular, we shall investigate when the $\aleph_0$-categoricity of a semigroup passes to subsemigroups, quotients and certain direct sums. This is certainly a popular path to take: we mention here  Grzegorczyk's handy result  that $\aleph_0$-categoricity (of a general structure) is   preserved by finite direct products \cite{Grzeg}. In \cite{Wasz}, Waszkiewicz and Weglorz 
showed that $\aleph_0$-categoricity of a structure is preserved by  Boolean extensions  by $\aleph_0$-categorical Boolean algebras, a result later generalized by Schmerl  \cite{Schmerl78} to filtered Boolean extensions. In \cite{Sab} Sabbagh proved that the group GL$_n(R)$ of invertible $n\times n$ matrices over an $\aleph_0$-categorical ring $R$ inherits $\aleph_0$-categoricity;  the corresponding result for the semigroup 
$M_n(R)$ of all $n\times n$ matrices follows easily from   Theorem~\ref{RNT}.

The final two sections  fit into the `classification' approach: determining the $\aleph_0$-categoricity of semigroups in certain classes built from $\aleph_0$-categorical components. In particular,  we classify $\aleph_0$-categorical 0-direct unions and certain
$\aleph_0$-categorical semidirect products, including the case where the semigroup being acted upon is a finite semilattice.  A number of known classifications will be of use in  our work, including the $\aleph_0$-categoricity of linear orders  \cite{Ros69}, which serve as examples of $\aleph_0$-categorical semilattices. For algebraic structures, the difficulty in achieving full classifications has long been apparent, although as indicated above, significant results are available for groups and rings. The former are of particular importance to this paper, since  maximal subgroups of $\aleph_0$-categorical semigroups are $\aleph_0$-categorical.

 Throughout the paper we develop tools  for ascertaining  $\aleph_0$-categoricity of semigroups, built on Theorem~\ref{RNT} with increasing degrees of complexity, which are made use of as follows. In Section~\ref{sec:first} we show that any \ale\, semigroup is periodic with bounded index and period, and it can be cut into \ale\, `slices'  which satisfy the additional property of being {\em characteristically (0-)simple}. We also prove the existence of an \ale\, nil semigroup that is not nilpotent, a situation that contrasts to that in ring theory.  In Sections~\ref{S rel char}  and
\ref{Sec princ} we develop a notion for  subsemigroups of being {\em relatively characteristic}, which is somewhat weaker than the standard notion of being characteristic 
(i.e. preserved by all automorphisms). We are then able to demonstrate how  $\aleph_0$-categoricity  is inherited by maximal subgroups, principal factors, and certain other quotients. The main result of Section~\ref{Sec princ} shows that a 
Brandt semigroup $\mathcal{B}^0(G;I)$ is \ale \, if and only if $G$ is an \ale \, group.
In Section~\ref{Sec 0-direct}   we consider direct sums and 0-direct sums, determining when a direct sum of finite monoids or semigroups is \ale\, and, via an analysis of 0-direct sums, when a primitive inverse semigroup is \ale. Finally in Section~\ref{sec:last} we examine how $\aleph_0$-categoricity interacts with semidirect products, and with the construction of McAlister $\mathcal{P}$-semigroups $\mathcal{P}=\mathcal{P}(G,\mathcal{X},\Y)$ in terms of the $\aleph_0$-categoricity of $G,\X$ and $\Y$, showing in particular that if $\X$ is finite then the
$\aleph_0$-categoricity of $\mathcal{P}$ depends only on $G$.

We denote the set of natural numbers (containing 0) by
 $\N$ ($\N^0$). This article does not require any background in model theory, but we refer the reader to \cite{Hodges93} for an  introductory study. 
 
\section{First examples of $\aleph_0$-categorical semigroups} \label{sec:first}
 
As commented in the Introduction, our  main tool in  determining
$\aleph_0$-categoricity is Theorem~\ref{RNT}.
 An immediate consequence worth highlighting is:
 
 \begin{corollary}\label{cor:finite}  Finite semigroups are  $\aleph_0$-categorical. 
 \end{corollary}

Recall that a  semigroup $S$  is  \textit{uniformly locally finite} (ULF) if there exists a function $f:\mathbb{N}\rightarrow \mathbb{N}$ such that for every subsemigroup $T$ of $S$, if $T$ has a generating set of cardinality at most $n$, then $T$ has cardinality at most $f(n)$. Rosenstein~\cite[Theorem 16]{Rosenstein73} showed that an
$\aleph_0$-categorical group is ULF, a result later proved for general structures, and quoted here for semigroups.

\begin{proposition}\cite[Corollary 7.3.2]{Hodges93} \label{cat ulf} An $\aleph_0$-categorical semigroup  is ULF. 
\end{proposition}

\begin{corollary}\label{cor:indexperiod} An $\aleph_0$-categorical semigroup $S$ is periodic, with bounded index and period. Consequently, $E(S)\neq \emptyset$ and $\dee=\jay$.
\end{corollary}

We note that the converse to  Proposition \ref{cat ulf} need not hold in general, as will be evident when we consider semilattices. However, a converse holds if we restrict our attention to homogeneous semigroups, where a semigroup is {\em homogeneous}  if every isomorphism between finitely generated  subsemigroups extends to an automorphism of the semigroup.
Again we quote a result for semigroups that is true for more general structures.

\begin{proposition} \cite[Corollary 7.4.2]{Hodges93} A homogeneous ULF semigroup is $\aleph_0$-categorical. 
\end{proposition} 

Since McLean showed  \cite{McLean} that any band is ULF, we immediately have:

\begin{corollary}\label{cor:homog} Homogeneous bands and homogeneous semilattices are  $\aleph_0$-categorical. \end{corollary}

 The homogeneity of both bands and inverse semigroups was studied by the second author, with results appearing in \cite{Quinn16} (where a complete characterisation 
of homogeneous bands is given)  and \cite{Quinn17}, respectively.  
Complete characterisations of homogeneous semilattices appear in  \cite{Droste,DrosteTruss}, and of $\aleph_0$-categorical linear orders in \cite{Ros69}.
It follows from these papers that
\[ \Q_1\cup\{  x_1,x_2\}\cup  \Q_2\]
is not homogeneous but is $\aleph_0$-categorical, 
where $\Q_1,\Q_2$ are copies of the linearly ordered set $\Q$ and
\[q_1<x_1<x_2<q_2\]
for all $q_i\in \Q_i, i=1,2$. Further, an $\omega$-chain
\[x_1<x_2<\hdots\]
or an inverse $\omega$-chain
\[x_1>x_2>\hdots\]
are not $\aleph_0$-categorical. This contrasts starkly with the fact that an abelian group is 
$\aleph_0$-categorical if and only if it is periodic of bounded index \cite{Rosenstein73}. 

\bigskip
At this stage it is convenient to develop some notation to help us implement the RNT in practice. Given a semigroup $S$ and  pair $\underline{a}=(a_1, \dots, a_n), \, \underline{b}=(b_1, \dots ,b_n)$ of $n$-tuples of $S$, then we say that $\underline{a}$ is \textit{automorphically equivalent to/has the same $n$-automorphism type as $\underline{b}$} (in $S$) if there exists an automorphism $\phi$ of $S$ such that $\underline{a}\phi=\underline{b}$ (so that  $a_i\phi = b_i$ for each $i$). We denote this relation on $S^n$ by $\underline{a}\sim_{S,n} \underline{b}$, using the same notation for the restriction  to subsets of $S^n$. Hence, by the RNT, to prove that $S$ is $\aleph_0$-categorical it suffices to show that, for each $n$, there exists a finite list of elements of $S^n$ such that every element of $S^n$ is automorphically-equivalent to an element of the list; equivalently, in any countably infinite list of elements of $S^n$, we can find two distinct members that are automorphically-equivalent.
We augment our notation as follows. Suppose that 
$\underline{X}$ is a finite tuple of elements of $S$.  Let Aut$(S;\underline{X})$ denote the subgroup of Aut($S$) consisting of those automorphism which fix $\underline{X}$. We say that $S$ is \textit{$\aleph_0$-categorical over $\underline{X}$}  if Aut$(S;\underline{X})$ has only finitely many orbits in its action on $S^n$ for each $n\geq 1$. We denote the resulting equivalence relation on $S^n$ as $\sim_{S,\underline{X},n}$, and a pair of $\sim_{S,\underline{X},n}$-equivalent $n$-tuples are said to be \textit{automorphically equivalent over $\underline{X}$}.

 Lemma~\ref{cat over finite} is a simple generalisation of \cite[Exercise 7.3.1]{Hodges93}, and follows immediately from the RNT.

\begin{lemma} \label{cat over finite} Let $S$ be a semigroup and $\underline{X}$ a finite tuple of elements of $S$. 
For any subset $T$ of $S$, we have that $|T^n/\sim_{S,n}|$ is finite for all $n\geq 1$ if and only if $|T^n/\sim_{S,\underline{X},n}|$ is finite for all $n\geq 1$. In particular, $S$ is $\aleph_0$-categorical if and only if $S$ is $\aleph_0$-categorical over $\underline{X}$. 
\end{lemma} 

\begin{example}\label{null cat} Consider the countably infinite null semigroup $N$, with multiplication $xy=0$ for all $x,y\in N$. Since any permutation of the non-zero elements gives an isomorphism,  it is clear that $N$ is homogeneous. Clearly $N$ is ULF, so that it is also $\aleph_0$-categorical. Indeed, $N$ provides a good illustration of the RNT. A pair of $n$-tuples $\underline{a}=(a_1, \dots, a_n)$ and $ \underline{b}=(b_1, \dots ,b_n)$ are automorphically equivalent if and only if  the positions  of the 0 entries are the same,
and for any $1\leq i,j\leq n$ we have $a_i=a_j$ if and only if $b_i=b_j$. Since there are finitely many choices for each of these conditions, it follows that $N$ is $\aleph_0$-categorical by the RNT.   
\end{example} 

%


It is worth formalising the points raised in the above argument, as they will be used throughout this paper. They are  based on the following lemma, which may be proven by a simple counting argument. 

\begin{lemma}\label{counting bits} Let $X$ be a set and $\gamma_1,\dots,\gamma_r$ be a finite list of equivalence relations on $X$ with $\gamma_1 \, \cap \,  \gamma_2 \,  \cap \cdots \cap \gamma_r$ contained in an equivalence relation $\sigma$ on $X$. Then 
\[ |X/\sigma| \leq \prod_{1\leq i \leq r} |X/\gamma_i|.  
\] 
\end{lemma} 

We use the RNT in conjunction with Lemma \ref{counting bits} to prove that a semigroup $S$ is $\aleph_0$-categorical in the following way. 

\begin{corollary} \label{cor:themethod} Let $S$ be a semigroup and for each $n\geq 1$, let $\gamma_1,\dots,\gamma_{r(n)}$ be a finite list of equivalence relations on $S^n$ such that $S^n/\gamma_i$ is finite for each $1\leq i \leq {r(n)}$ and 
\[ \gamma_1\cap \gamma_2 \cap \cdots \cap \gamma_{r(n)} \subseteq \, \sim_{S,n}. 
\] 
Then $S$ is   $\aleph_0$-categorical.
 \end{corollary}
 
Where  no confusion is likely to  arise, we may refer to equivalences to which we apply  Corollary~\ref{cor:themethod} in a less formal way, as follows. Suppose that we have an equivalence relation on 
$n$-tuples of a semigroup $S$, that arises from different ways in which a given condition may be fulfilled; if there are only finitely many classes of the equivalence, then we say the condition has {\em finitely many choices}. 

\begin{example}\label{natural} (i) If  $S$ is a semigroup with zero (as in Example~ \ref{null cat}), the equivalence 
$\sim_0$ on $S^n$ defined by the rule that \[(a_1,\hdots, a_n)\sim_0 
(b_1,\hdots, b_n)\Leftrightarrow \{ i:a_i=0\}=\{ i:b_i=0\}\]
corresponds to the condition  on $n$-tuples that they have the non-zero entries in the same positions, and this condition has $2^n$ choices.  
  
 (ii) Again as in Example~ \ref{null cat}, given a set $X$, we may impose a condition on a pair of $n$-tuples of $X$ which states that if a pair of entries in one of the tuples are equal then the same is true for the other tuple. Formally, we define an equivalence $\natural_{X,n}$ on $X^n$ by
\begin{equation} \label{natural n}   (a_1,\dots,a_n) \, \natural_{X,n} \, (b_1,\dots,b_n) \text{ if and only if } [a_i=a_j \Leftrightarrow b_i=b_j, \text{ for each } i,j].
\end{equation} 
It is clear that a pair of $n$-tuples $\underline{a}$ and $\underline{b}$ are $\natural_{X,n}$-equivalent if and only if there exists a bijection $\phi:\{a_1,\dots,a_n\} \rightarrow \{b_1,\dots,b_n\}$ such that $a_i\phi=b_i$.  Moreover, the number of $\natural_{X,n}$-classes of $X^n$ is equal to the number of ways of partitioning a set of size $n$, which is  called the \textit{$n$th Bell number}, denoted $B_n$ (for a formulation, see \cite{Rota}). In particular $B_n$ is finite, for each $n\geq 1$. Note also that if $S$ is a semigroup then  
 \[ \underline{a} \, \sim_{S,n} \, \underline{b} \Rightarrow \underline{a} \, \natural_{S,n} \, \underline{b}. 
\]  
\end{example}

We will see in this paper that $\aleph_0$-categoricity  `works well' in conjunction with 
fixing finite sets of elements within semigroups.   In particular, in order  to prove that $S$ has finitely many $n$-automorphism types,   it suffices to  consider $n$-tuples of $S
\setminus T$, where $T$ is finite. 


\begin{proposition}\label{excluded} Let $S$ be a semigroup and $T$ a finite subset of $S$. Then $S$ is  $\aleph_0$-categorical if and only if  $|(S\setminus T)^n /\sim_{S,n}|$ is finite for each $n\geq 1$.  
\end{proposition} 

\begin{proof} If $S$ is $\aleph_0$-categorical then $|S^n/\sim_{S,n}|$ is finite by the RNT, and thus so is $|(S\setminus T)^n /\sim_{S,n}|$. 

For the converse, we begin by fixing  some notation. Let $A$ be a subset of $S$ and $\underline{s}=(s_1,\dots,s_n)$ an $n$-tuple of $S$. Then we let 
\[ \underline{s}[A]:=\{k\in \{1,\dots,n\}: s_k\in A\}
\] 
be the set of positions of entries of $\underline{s}$ which lie in $A$. If $\underline{s}[A]=\{k_1,\dots,k_r\}$ is such that $k_1<k_2<\cdots <k_r$ then we obtain an $r$-tuple of $A$ given by  
\[ \underline{s}^A:=(s_{k_1},\dots,s_{k_r}). 
\] 
Let $T=\{t_1,\dots,t_r\}$ and take $\underline{T}=(t_1,\dots,t_r)\in S^r$.
 Let $\underline{a}$ and $\underline{b}$ be $n$-tuples of $S$ under the conditions that 
\begin{enumerate}[label=(\arabic*)]
\item $\underline{a}[T]=\underline{b}[T]$ with $\underline{a}^T = \underline{b}^T$, 
\item   $\underline{a}^{S\setminus T}$ and $\underline{b}^{S\setminus T}$ are automorphically equivalent over $\underline{T}$.\end{enumerate} 
Condition (1) has $(|T|+1)^n$ choices, which is finite since $T$ is. Each $|(S\setminus T)^m /\sim_{S,m}|$ is finite by our hypothesis, and so $|(S\setminus T)^m/\sim_{S,\underline{X},m}|$ is also finite for each $m\in \mathbb{N}$ by Lemma \ref{cat over finite}. Hence condition (2) has finitely many choices, and the total number of choices is therefore finite. By condition (2) there exists $\phi\in \text{Aut}(S;\underline{T})$ with $\underline{a}^{S\setminus T}\phi=\underline{b}^{S\setminus T}$. Since $t_i\phi=t_i$ for each $1\leq i \leq r$ we have $\underline{a}^T\phi=\underline{a}^T=\underline{b}^T$, and it follows that $\underline{a}\phi=\underline{b}$. The result is then immediate from Lemma \ref{counting bits}. 
\end{proof}

For a semigroup $S$ we let $S^{\underline{1}}$ \footnote{By convention, $S^1$ denotes $S$ with {\em an identity adjoined if necessary}.} ($S^0$) denote $S$ with an identity (zero) adjoined (whether or not $S$ already has such an element). The next result follows from
Proposition~\ref{excluded} and the fact that automorphisms of $S^{\underline{1}}$ and
$S^0$ are exactly extensions of automorphisms of $S$. 

\begin{corollary}\label{cor:idzero} The following are equivalent for any semigroup $S$:
\begin{enumerate}\item $S$ is  $\aleph_0$-categorical;
\item $S^0$ is $\aleph_0$-categorical;
\item $S^{\underline{1}}$ is $\aleph_0$-categorical.
\end{enumerate}
\end{corollary}

In a similar fashion to that in Corollary~\ref{cor:idzero} we can build new 
$\aleph_0$-categorical semigroups from given ingredients, provided the ingredients interact in a relatively simplistic way (see Proposition~\ref{prop:chain} below). 

We now introduce an important notion for $\aleph_0$-categorical semigroups. 

\begin{definition} A subset $A$ of a semigroup $S$ \textit{characteristic} if  it is invariant under automorphisms of $S$, that is, $A\phi=A$ for all $\phi\in \text{Aut}(S)$. 
\end{definition} 
 Clearly any subset is characteristic if and only if it is a union of $\sim_{S,1}$-classes, and if $A$ is a characteristic subset of a semigroup $S$ then $\langle A \rangle$ is a characteristic subsemigroup of $S$. 

Let $S$ be either a semigroup with zero $0$, or a ring,  and let $n\in \N$. We say that $S$ is {\em nil}
of degree $n$ 
if for all $a\in S$ we have $a^n=0$, and $S$ is {\em nilpotent} of degree $n$ if $S^n=0$. 

\begin{corollary} \label{cor:ch}  Let $S$ be  $\aleph_0$-categorical. Then
\begin{enumerate} \item  there are finitely many characteristic subsets of $S$;
\item any  characteristic subsemigroup of $S$ is $\aleph_0$-categorical;
\item any ideal $S^m$ is characteristic and hence $\aleph_0$-categorical;
\item  for some $n\geq 1$ we have $S^n=S^{n+1}$, so that $S^n=S^m$ for all $m\geq n$;
\item with $n$ as in (3), for any $k<{\ell}\leq n$ we have that $S^{\ell}$ is an ideal of $S^k$ and the Rees quotient $S^k/S^{\ell}$ is $\aleph_0$-categorical;
\item with $n$ as in (3),  $S/S^n$ is $\aleph_0$-categorical and nilpotent of degree $n$.
\end{enumerate}
\end{corollary}
\begin{proof} (1) follows from the fact that a subset is characteristic if and only if it is a union of $\sim_{S,1}$ classes and (2) is immediate from the definition of 
characteristic subsemigroup.  (3) is clear and then  (4) is immediate from (1), (3), and the fact that  $S^n\supseteq S^{n+1}$. 

For (5), observe that $S^{\ell}$ is an ideal of $S^k$. To see that it is
$\aleph_0$-categorical, let $m\in\N$ and consider a list of
$m$-tuples of elements of $S^k/S^{\ell}$. By Proposition~\ref{excluded} 
we may assume all of these elements are non-zero, and we may thus identify them with 
elements of $S^k$. Since $S^k$ is $\aleph_0$-categorical we may find a distinct pair 
$(a_1,\hdots, a_m)$ and $(b_1,\hdots, b_m)$ in our list and $\phi\in\Aut S^k$ such that
$a_i\phi=b_i$ for $1\leq i\leq m$. It is easy to see that $\phi$ induces an automorphism
$\phi'$ of 
$S^k/S^{\ell}$, and regarded as $m$-tuples of $S^k/S^{\ell}$, we have
$(a_1,\hdots, a_m)\phi'=(b_1,\hdots, b_m)$. Hence by Proposition~\ref{excluded},
$S^k/S^{\ell}$ is $\aleph_0$-categorical.

For (6), observe  $S/S^n$ is  nilpotent of degree $n$. The rest of the statement follows from (5). 
\end{proof}

Corollary~\ref{cor:ch} shows that any \ale \, semigroup is associated with a nilpotent one.
A major result for \ale \, rings states that any \ale \, nil ring of degree $n$  is nilpotent of degree $n$ \cite{Cherlin}. We show that the corresponding result is not true for semigroups, by constructing a countably infinite $\aleph_0$-categorical commutative  semigroup $S$ such that $S$ is nil of degree 2 and $S=S^2$, so that certainly $S$ is not non-nilpotent. 

A commuative  semigroup $S$, nil of degree 2 such that $S=S^2$ is called a \textit{zs-semigroup}. Some progress has been made in understanding the structure of zs-semigroups, including \cite{Jezek} and \cite{Flaska}. 
In \cite{Flaska}, a simple example of a zs-semigroup is constructed, which is very similar to that given below. However, we need to start with a countable atomless Boolean algebra
in order to ensure the resulting semigroup is \ale.

\begin{remark}\label{boolean} There exists a unique, up to isomorphism, countable atomless Boolean algebra $B$
\cite[Theorem 10]{givanthalmos} \footnote{We thank Prof. John Truss of the University of Leeds for bringing this example to our attention.}. Since atomless Boolean algebras are axiomatisable, it follows that $B$ is \ale.
We can  construct $B$ in a number of ways, including the Lindenbaum algebra of propositional logic \cite[Chapter 6]{Landman}. For our purposes it is convenient to  use the construction given in \cite[Corollary 23]{Gratzer} via certain subsets of $[0,1]$. Let $B$ be the set of all subsets of $[0,1]\cap \mathbb{Q}$ of the form 
\[ (a_0]\cup (b_1,a_1]\cup \cdots \cup (b_{n-1},a_{n-1}], 0<a_0<b_1<a_1<\cdots <b_{n-1}<a_{n-1},
\] 
where $a_0, b_1,a_1,\hdots, a_{n-1}\in \mathbb{Q}$. It is clear that $B$ forms a subalgebra of the Boolean algebra of subsets of $(0,1]\cap \mathbb{Q}$ and is atomless.  Moreover, if $A = (a_0]\cup (b_1,a_1]\cup \cdots \cup (b_{n-1},a_{n-1}]$ is a non-empty element of $B$ then, taking any $x\in (b_1,a_1)\cap \mathbb{Q}$, we have that 
\[ A=\big((a_0]\cup (b_1,x]\big) \cup \big((x,a_1]\cup \cdots \cup (b_{n-1},a_{n-1}]\big) \]
and
\[ \big((a_0]\cup (b_1,x]\big) \cap \big((x,a_1]\cup \cdots \cup (b_{n-1},a_{n-1}]\big)=\emptyset
.\] 
\end{remark}

\begin{theorem} There is an \ale\, non-nilpotent, nil semigroup. 
\end{theorem}
\begin{proof} We show there is an \ale\, zs-semigroup that is not nilpotent.

Let $B=(B,\wedge, \lor, 0,1)$ be the unique countable atomless Boolean algebra and let $B^*=B\setminus \{0\}$. 
Define an operation $+$ on $B^*$ by
\[ A+A' = \left\{
\begin{array}{ll}
A \lor A' & \text{if       } A\wedge A'=0,\\
1 & \text{else. } 
\end{array}
\right. 
\]
Since $\lor$ is an associative operation and $\wedge$ distributes over $\lor$, we have for any $A_1,A_2,A_3\in B^*$, 
\begin{align*}
(A_1+A_2)+A_3 & = \left\{
\begin{array}{ll}
(A_1\lor A_2)+A_3 & \text{if       } A_1\wedge A_2=0,\\
1+A_3 & \text{else, } 
\end{array}
\right.  \\
& = \left\{
\begin{array}{ll}
(A_1\lor A_2)\lor A_3 & \text{if       } A_1\wedge A_2=0 \text{ and } (A_1 \lor A_2)\wedge A_3=0,\\
1 & \text{else, } 
\end{array}
\right.   \\
& = \left\{
\begin{array}{ll}
A_1\lor (A_2\lor A_3) & \text{if       } A_1\wedge A_2=0 \text{ and } (A_1 \wedge A_3) \lor (A_2\wedge A_3)=0,\\
1 & \text{else, } 
\end{array}
\right. \\
& = \left\{
\begin{array}{ll}
A_1\lor (A_2\lor A_3) & \text{if       } A_1\wedge A_2=0, A_1 \wedge A_3=0, \text{ and }  A_2\wedge A_3=0,\\
1 & \text{else, } 
\end{array}
\right.
\\
& = \left\{
\begin{array}{ll}
A_1\lor (A_2\lor A_3) & \text{if       } A_2\wedge A_3=0 \text{ and }  A_1 \wedge (A_2 \lor A_3)=0,\\
1 & \text{else, } 
\end{array}
\right. \\
& = \left\{
\begin{array}{ll}
A_1+ (A_2\lor A_3) & \text{if       } A_2\wedge A_3=0,\\
A_1+ 1 & \text{else, } 
\end{array}
\right. \\
& = A_1+(A_2+A_3). 
\end{align*}
Hence $(B^*,+)$ forms a semigroup, and is commutative by the commutativity of $\wedge$ and $\lor$. Note that  $A+1 = A\lor 1 = 1$ for all $A\in B^*$, so that 1 is the zero of $(B^*,*)$. If $A\in B^*$, then $A+A=1$ as $A\wedge A=A$, and so $(B^*,+)$ forms a nil semigroup of degree 2. Moreover,  by the last part of Remark~\ref{boolean} for 
any $A\in B^*$ we have  $A=A_1 \lor A_2$ for some $A_1,A_2\in B^*$ such that $A_1 \wedge A_2=0$. Hence $A=A_1+A_2$, and so $(B^*,+)$ forms  a zs-semigroup. Let us denote this semigroup by $[B]$. It is clear that if $\theta$ is an isomorphism of $B$, then 
$\theta|_{B^*}$ is an isomorphism of $B^*$. Since $B$ is \ale, so also then is $[B]$.
\end{proof} 

We can say a little more: it is easy to build an example of an \ale\,  commutative nil semigroup that is nil of degree 2 and nilpotent but {\em not} nilpotent of degree 2. 

\begin{example} Let $A$ be countably infinite, and let $u,0$ be distinct symbols not in
$A$. Let $C=A\cup \{ 0,u\}$ and define a binary operation on $C$ by letting the only non-zero products be $ab=u$ where $a,b\in A$ and $a\neq b$. It is easy to see that $C$ is a commutative semigroup, nil of degree 2 and nilpotent of degree 3. That $C$ is \ale\, follows easily from Proposition~\ref{excluded}. 
\end{example}
 
We say that a semigroup $S$ is {\em characteristically simple} if it has no  characteristic ideals other than $\emptyset$ or itself.  Similarly, we say that a semigroup $S$ with $0$  is {\em characteristically 0-simple} if it has no   characteristic ideals, other than $\{ 0\}$ and itself. 

\begin{proposition} \label{prop:thetower} Let $S$ be an \ale\, semigroup. Then $S$ is the union of a finite chain of characteristic  subsemigroups
\[S=S_0\supset S_1\supset\hdots \supset S_n\]
such that for $0\leq i\leq n-1$, $S_{i-1}$ is a characteristic ideal of $S_i$ and the Rees quotients $S_i/S_{i+1}$ are \ale\,\,  and characteristically 0-simple, 
and $S_n$ is characteristically simple. 
\end{proposition}
\begin{proof}  For an \ale\, semigroup $S$, let $\tau(S)$ denote $|S/\sim_{S,1}|$.  Let $U$ be  a characteristic subsemigroup of $S$. Notice that for any $n\in\N$ and $\underline{u},\underline{v}\in U^n$, if $\underline{u}\sim_{S,n} \underline{v}$, then
 $\underline{u}\sim_{U,n} \underline{v}$, since $U\phi=U$ for any $\phi\in\Aut S$. From the above, if $U$ is a characteristic subsemigroup of $S$,
and $U\neq S$,  then $U$ is \ale\, and $\tau(U)<\tau(S)$. 

 We proceed by induction on $\tau(S)$. If $\tau(S)=1$, then certainly there are no proper characteristic ideals of 
$S$, so that the result is true with $n=0$. 

Suppose now that for any \ale \,  semigroup $T$ with $\tau(T)<\tau(S)$ the result holds. Let $T$ be a maximal  proper characteristic ideal of $S$.  If $T=\emptyset$ 
 then we are done. Suppose therefore that $T\neq\emptyset$;
 the proof that 
 $S/T$ is \ale\, follows as in (5) of Corollary~\ref{cor:ch} (see also (1) of Corollary~\ref{f.g cong cat}). 
If $U$ is a proper characteristic ideal of $S/T$, then either
$U=\{ 0\}$, or $U\setminus \{ 0\}\cup T$ is an ideal of $S$. Since $U\setminus \{0\}$ is a union of $\sim_{S/T,1}$-classes and hence of $\sim_{S,1}$-classes, we have
that $(U\setminus\{ 0\})\cup T$ is a characteristic ideal of $S$ strictly containing $T$, a contradiction. Thus $S/T$ is characteristically 0-simple. 

From the first part of the proof we have $\tau(T)<\tau(S)$, so that, applying the result for
$T$, we deduce the required sequence of ideals  for $S$.
\end{proof}

The following example is clear. 

\begin{example} \label{Y char} Let $S$ be a semigroup. Then the following subsets (where they exist) are characteristic:
\[ E(S), \langle E(S)\rangle, \{ 1\}, \{ 0\}, \Reg(S), \langle \Reg(S)\rangle,\]
where Reg$(S)$ is the set of regular elements of $S$. If $S$ is commutative, Reg$(S)$ forms a semilattice of abelian groups; we address the $\aleph_0$-categoricity of Clifford semigroups in the sequel \cite{GouldQuinn}. Proposition~\ref{prop:chain} below gives a taster of the results for Clifford semigroups, in the special case where the $S_i$ are groups and the connecting homomorphisms are trivial.
\end{example}

\begin{proposition}\label{prop:chain} Let $S=\bigcup_{i\in Y}S_i$ be a finite chain of semigroups such that for any $i>j, s_i\in S_i, s_j\in S_j$ we have
$s_is_j=s_j=s_js_i$. If each $S_i$ is $\aleph_0$-categorical then $S$ is $\aleph_0$-categorical. Moreover, if each $S_i$ is characteristic (for example, if it is a non-trivial group) then the converse holds.
\end{proposition}
\begin{proof} Let us refer to the $S_i$ ($i\in Y$) as the components of $S$. 

Suppose each $S_i$ is $\aleph_0$-categorical. Let $n\in\N$ and notice that in any infinite list of elements of $S^n$ we can pick a sublist $(a_1^i,\hdots, a_n^i)$ such that for any $1\leq \ell\leq n$ the elements
$a_{\ell}^1,a_{\ell}^2,\hdots$ all lie in the same component of $S$. Without loss of generality, suppose that
$a_1^k,\hdots, a^k_{j_1}\in S_{i_1}, a^k_{j_1+1},\hdots, a^k_{j_2}\in S_{i_2},\hdots,
a^k_{j_{u-1}+1},\hdots , a^k_{n}\in S_{i_u}$. Since each $S_k$ is
$\aleph_0$-categorical, we may find an $i<j$ and
$\phi_{\ell}\in \Aut S_{i_{\ell}}, 1\leq \ell\leq u$, such that
$\phi'=\bigcup_{1\leq \ell\leq u}\phi_{\ell}$ takes $(a^i_1,\hdots, a^i_n)$ to
$(a^j_1,\hdots, a^j_n)$. For any $t\in Y\setminus \{ i_1,\hdots, i_u\}$, let $\phi_t=I_{S_t}$.
It is easy to see that $\phi=\bigcup_{i\in Y}\phi_i$ lies in $\Aut S$ and clearly takes 
$(a^i_1,\hdots, a^i_n)$ to
$(a^j_1,\hdots, a^j_n)$. 

The converse is clear. 
\end{proof}

Finally in this section we make a comment  concerning chains of (one-sided) ideals of $S$.
 Recall that in a partially ordered set $L$, an element $u$ {\em covers}
an element $v$, written $v\prec u$,  if $v<u$ and for all $w$ with $v\leq w\leq u$ we have
$v=w$ or $w=u$. If we have a chain of elements in $L$ such that each element covers its predecessor, then we call this a {\em covering chain}.

\begin{lemma} Let $S$ be an  $\aleph_0$-categorical semigroup in which the principal 
right (left, two-sided) ideals form a chain. Then there are no infinite ascending or descending  covering chains of principal right (left, two-sided) ideals.
\end{lemma}
\begin{proof} We argue for ascending chains of principal right ideals, the other cases being similar. 

Suppose that $a_1,a_2,\hdots\in S$ and
\[a_1S^1\prec a_2S^1\prec \hdots .\]
Then $(a_1,a_i)$ for $i\in\N$ lie in different 2-automorphism types, a contradiction. 
\end{proof}

As the case of a dense linear order shows, we cannot expect to have full ascending or descending chain conditions on ideals in $\aleph_0$-categorical semigroups.

\section[Relatively characteristic]{Inherited categoricity} \label{S rel char} 

We remarked in Corollary~\ref{cor:ch} that $\aleph_0$-categoricity is inherited by characteristic subsemigroups. 
We note that $\aleph_0$-categoricity is not inherited by every subsemigroup, and an example for groups is given by Olin in \cite{Olin}. 
However, the condition that a subsemigroup be
characteristic to inherit $\aleph_0$-categoricity is too restrictive, since many key subsemigroups, such as maximal subgroups and principal ideals, are not necessarily characteristic. The components in 
a finite chain of groups as in Proposition~\ref{prop:chain} are, but this relies on the chain being finite. We thus study a weaker condition for a subsemigroup that still guarantees the preservation of  $\aleph_0$-categoricity. 

 \begin{definition}\label{fprc} Let $S$ be a semigroup and, for some fixed $t\in \mathbb{N}$, let  $\{\underline{X}_i:i\in I\}$ be a collection of $t$-tuples of $S$. Let $\{A_i:i\in I\}$ be a collection of subsets of $S$ with the property that for any automorphism $\phi$ of $S$ such that there exists $i,j\in I$ with $\underline{X}_i \phi = \underline{X}_j$, then $\phi|_{A_i}$ is a bijection from $A_i$ onto $A_j$.  Then we call $\mathcal{A}=\{(A_i,\underline{X}_i):i\in I\}$ a \textit{system of $t$-pivoted pairwise relatively characteristic ($t$-pivoted p.r.c.) subsets} (or, {\em subsemigroups}, if each $A_i$ is a subsemigroup) of $S$. The $t$-tuple $\underline{X}_i$ is called the \textit{pivot} of $A_i$ ($i\in I$). 
 If $|I|=1$ then, letting $A_1=A$ and $\underline{X}_1=\underline{X}$, we write $\{(A,\underline{X})\}$ simply as $(A,\underline{X})$, and call $A$ an \textit{$\underline{X}$-pivoted relatively characteristic ($\underline{X}$-pivoted r.c.) subset/subsemigroup} of $S$. 
 \end{definition} 

Clearly if  $\{(A_i,\underline{X}_i):i\in I\}$ forms a system of $t$-pivoted p.r.c. subsets of $S$ and $J$ is a subset of $I$ then $\{(A_j,\underline{X}_j):j\in J\}$ is also a system of $t$-pivoted  p.r.c. subsets of $S$. In particular, each $A_i$ is an $\underline{X}_i$-pivoted r.c. subset of $S$. 
Moreover, if $A$ is an  $\underline{X}$-pivoted r.c. subset of $S$ then $A$ is a union of orbits of the set of automorphisms of $S$ which fix $\underline{X}$, since
if $a\in A$ and $\phi\in \text{Aut}(S)$ fixes $\underline{X}$ then $A\phi=A$, so that $a\phi\in A$. 

Definition \ref{fprc} has strong links with the model theoretic  concept of \textit{definability}, and we refer the reader to the introduction of \cite{Evans} for a background into these links.
 In fact much of the work in this section could be given in terms of definable sets, but in keeping with our algebraic viewpoint it is more natural to use Definition \ref{fprc}. 
     
\begin{lemma}\label{rel-char equiv}  Let $S$ be a semigroup and, for some fixed $t\in \mathbb{N}$, let  $\{\underline{X}_i:i\in I\}$ be a collection of $t$-tuples of $S$. Then for any collection $\{A_i:i\in I\}$ of subsets of $S$, the following are equivalent: 
\begin{enumerate}[label=(\roman*)]
\item  $\{(A_i,\underline{X}_i):i\in I\}$ is a system of $t$-pivoted p.r.c. subsets/subsemigroups of $S$; 
 \item if $\phi\in \Aut(S)$ is such that there exists $i,j\in I$ with $\underline{X}_i\phi = \underline{X}_j$, then $A_i\phi\subseteq A_j$. 
  \end{enumerate}
\end{lemma} 

\begin{proof} This follows immediately from applying the definitions, and the fact that if $\phi\in\Aut A$ and  $\underline{X}_i\phi = \underline{X}_j$, then 
 $\phi^{-1}\in \Aut A$ and  $\underline{X}_j\phi^{-1} = \underline{X}_i$. 
\end{proof}  
 
Consequently,  if  $\{(A_i,\underline{X}_i):i\in I\}$ is a system of $t$-pivoted p.r.c. subsets of a semigroup $S$ then $\{(\langle A_i \rangle,\underline{X}_i):i\in I\}$ forms a system of $t$-pivoted p.r.c. subsemigroups of $S$. 
For if $\phi\in \text{Aut}(S)$ is such that $\underline{X}_i\phi=\underline{X}_j$ for some $i,j\in I$ then $A_i\phi=A_j$, and so $\langle A_i \rangle \phi \subseteq \langle A_j \rangle$. The result follows by Lemma \ref{rel-char equiv}.

\begin{notation} Given a pair of tuples $\underline{a}=(a_1,\dots,a_n)$ and $\underline{b}=(b_1,\dots,b_m)$, we denote $(\underline{a},\underline{b})$ as the $(n+m)$-tuple given by 
\[ (a_1,\dots,a_n,b_1,\dots,b_m).
\]
\end{notation}

\begin{proposition}\label{pairwise-rel-char} Let $S$ be an $\aleph_0$-categorical semigroup and $\{(A_i,\underline{X}_i):i\in I\}$ be a system of $t$-pivoted p.r.c. subsets of $S$. Then $\{|A_i| : i \in I\}$ is finite. If, further, each $A_i$ forms a 
subsemigroup of $S$, then $\{A_i: i \in I\}$ is finite, up to isomorphism, with each $A_i$ being $\aleph_0$-categorical.
\end{proposition} 

\begin{proof} Suppose for some $i\neq j$ we have $\underline{X}_i \, \sim_{S,t} \, \underline{X}_j$ via $\phi\in \text{Aut}(S)$, say. Then $A_i\phi=A_j$ and it follows that both $|\{|A_i|:i\in I\}|$ and number of non-isomorphic elements of $\{A_i:i\in I\}$ is bound by the number of $t$-automorphism types of $S$, which is finite by the $\aleph_0$-categoricity of $S$.

Suppose $A_i$ forms a subsemigroup of $S$. Let $\underline{X}_i=(x_{i1},\dots,x_{it})$, and suppose $\underline{a}=(a_1,\dots,a_n)$ and $ \underline{b}=(b_1,\dots,b_n)$ are a pair of $n$-tuples of $A_i$ such that $(\underline{a},\underline{X}_i) \, \sim_{S,n+t} \, (\underline{b},\underline{X}_i)$ via $\phi\in \text{Aut}(S)$, say. Then $\underline{X}_i\phi=\underline{X}_i$ and so $\phi|_{ A_i }$ is an automorphism of $A_i$ as $(A_i,\underline{X}_i)$ is a $t$-pivoted r.c. subsemigroup. Moreover, $\underline{a}\phi|_{A_i} =\underline{a}\phi =  \underline{b}$ and so $\underline{a} \, \sim_{A_i,n} \, \underline{b}$. We have thus shown that 
\[ | A_i^n / \sim_{A_i,n}|  \leq |S^{n+t}/\sim_{S,n+t}|< \aleph_0
\] 
for each $n\geq 1$, since $S$ is $\aleph_0$-categorical. Hence $A_i$ is $\aleph_0$-categorical by the RNT.  
\end{proof}

\begin{corollary}\label{J} Let $S$ be an $\aleph_0$-categorical semigroup.  Then
there are only finitely many principal (left/right) ideals, up to isomorphism, and these are all $\aleph_0$-categorical. 
\end{corollary}
\begin{proof} We show that  $\{(S^1aS^1,a):a\in S\}$ forms a system of 1-pivoted p.r.c. subsets of $S$. To see this, let $\phi\in \text{Aut(S)}$ be such that $a\phi=b$, and let $x\in S^1aS^1$.
 Then there exists $u,v\in S^1$ with $x=uav$, and so by interpreting $1\phi$ as 1 we have
 \[ x\phi=(u\phi)(a\phi)(v\phi)=(u\phi)b(v\phi)\in S^1bS^1,
 \]
 and the result follows by Lemma \ref{rel-char equiv}. A similar result holds for principal left/right ideals. 
\end{proof}

Motivated by Green's relations, we now give a method for constructing systems of $t$-pivoted p.r.c. subsets of a semigroup via certain equivalence relations. 
Let $\phi:S\rightarrow T$ be an isomorphism between semigroups $S$ and $T$, and $\tau_S$ and $\tau_T$ be equivalence relations on $S$ and $T$, respectively. We call $\tau_S$ and $\tau_T$ \textit{preserved under $\phi$} if $a \, \tau_S \, b$ if and only if $a\phi \, \tau_T \, b\phi$ for each $a,b\in S$. This is clearly equivalent to 
\[  (x{\tau_S})\phi=(x\phi){\tau_T} \quad (\forall x\in S),
\] 
where $\phi$ is applied pointwise, and yet again to $\overline{\phi}:
S/\tau_S\rightarrow T/\tau_T$ given by
\[ [x]_{\tau_S}\overline{\phi}=[x\phi]_{\tau_T}  \quad (\forall x\in S)\]
being a well-defined bijection.
If $S=T$ then we say that $\tau_S$ is \textit{preserved under $\phi$}. 

Note that if $\tau$ is an equivalence relation on a semigroup $S$ then 
\[ \text{Aut}(S)[\tau]:=\{\phi \in \text{Aut}(S):\tau \text{ is preserved under } \phi\}
\] 
 is a subgroup of Aut($S$). 

If $\text{Aut}(S)=\text{Aut}(S)[\tau]$ then we call $\tau$ \textit{preserved under automorphisms} (of $S$).

\begin{example}\label{ex:congs}\begin{enumerate}\item For any semigroup $S$, if $U$ is a characteristic subset, then 
$\langle U\times U\rangle$
is preserved by automorphisms. 
\item  If $S$ is an inverse semigroup, then the least group congruence $\sigma$ on $S$ given by 
 \[ a \, \sigma \, b \Leftrightarrow (\exists e\in E(S)) \quad ea=eb,
 \] 
 is preserved by automorphisms. 
 \item If $\rho$ is a relation preserved by automorphisms, then so too are the congruences
 $\rho^{\sharp}=\langle \rho\rangle$ and $\rho^{\flat}$, where $\rho^{\flat}$ is the largest congruence contained in $\rho$. 
 \item If $S$ is an inverse semigroup then $\mu$, the maximum idempotent-separating congruence on $S$, is  preserved by automorphisms. 
 \end{enumerate} 
 \end{example} 
 \begin{proof} Statements (1) and (3) are clear. We remark that if $\sigma$ is given by the formula in (2), then
 $\sigma=\langle E(S)\times E(S)\rangle$, and for (4) we require the fact that  for any inverse semigroup $S$, we have $\mu=\mathcal{H}^{\flat}$ \cite[Proposition 5.3.7]{Howie94}. \end{proof}

The following lemma is then immediate from Proposition~\ref{pairwise-rel-char}.

\begin{lemma} \label{equiv classes rel} Let $S$ be a semigroup and $\tau$ be an equivalence relation on $S$, preserved under automorphisms of $S$. Then  $\{(x\tau,x):x\in S\}$ forms a system of 1-pivoted p.r.c. subsets of $S$, so that there are only finitely many cardinalities of $\tau$-classes and only finitely many subsemigroup $\tau$-classes, up to isomorphism.
\end{lemma} 

\begin{corollary}\label{cor:green} Let $S$ be an $\aleph_0$-categorical semigroup. Then for any of Green's relations $\mathcal{K}$, we have
\[|\{ |K_a|:a\in S\}|<\aleph_0.\]
Moreover, if $U$ is a transversal of the set of $\mathcal{K}$-classes that are subsemigroups, there are only finitely many $K_u$-classes $(u\in U)$, up to isomorphism, and each $K_u$ is $\aleph_0$-categorical. 

In particular, there are only finitely many maximal subgroups, up to isomorphism, and each of these is $\aleph_0$-categorical. 
\end{corollary}
\begin{proof} 
Each Green's relation is preserved by   automorphisms. 
Consequently, for any semigroup $S$ and any
$K\in \{ \ar,\el,\eh,\dee,\jay\}$, we have  $\{(K_a,a):a\in S\}$ as a  system of 1-pivoted p.r.c. subsets of $S$. The result then follows from Lemma~\ref{equiv classes rel}. 
\end{proof}

A similar statement to the above also holds for Green's *-relations, and Green's
$\widetilde{\phantom{\mathcal{R}}}$-relations \cite{gould:notes}. 
It is worth exercising  some caution here. In the corollary above the maximal subgroups are $\aleph_0$-categorical semigroups, while earlier investigations into the $\aleph_0$-categoricity of groups considered them as a set with a single binary operation, a single unary operation (inverse), and a single constant (the identity). However, since a semigroup automorphism of a group is necessarily a group automorphism, it follows from the RNT that our two concepts of $\aleph_0$-categoricity of a group coincide, and we can write \textit{$\aleph_0$-categorical group} without ambiguity.

Much like the situation with characteristic subsets, for results relating to inherited $\aleph_0$-categoricity of quotients we require only that congruences are preserved by all automorphisms fixing a finite number of elements. This leads us to the following definition.

\begin{definition} Let $\tau$ be an equivalence relation on a semigroup $S$ and $\underline{X}$ a  tuple of $S$. 
We say that  $\tau$ is  \textit{$\underline{X}$-relatively automorphism preserved ($\underline{X}$-r.a.p.) with pivot $\underline{X}$}, if whenever $\phi\in \text{Aut}(S)$ is such that $\underline{X}\phi=\underline{X}$, then $\phi$ preserves $\tau$. 
\end{definition}

We note that, as with $\underline{X}$-pivoted r.c. subsets, there exists connections between definable sets of ordered pairs of a semigroup and $\underline{X}$-r.a.p. equivalence relations. 

\begin{lemma} Let $S$ be a semigroup, let $\underline{X}\in S^t$ for some $t\in\N^0$,  and $\tau$  an $\underline{X}$-r.a.p. equivalence relation on $S$. For each $a\in S$, let $\underline{X}_a$ be the $(t+1)$-tuple given by $(\underline{X},a)$.  Then $\{(a\tau,\underline{X}_a):a\in S\}$ forms a system of $(t+1)$-pivoted p.r.c. subsets of $S$. 
\end{lemma}  

\begin{proof} Let $\phi$ be an automorphism of $S$ such that $\underline{X}_a\phi=\underline{X}_b$ for some $a,b\in S$. Then $\underline{X}\phi=X$ so that $\tau$ is preserved under $\phi$, and $a\phi=b$. Hence 
\[ (a{\tau})\phi=b\tau.
\] \end{proof}

Our next aim is to use the results above to assess when the $\aleph_0$-categoricity of a semigroup passes to its quotients.


\begin{proposition}
\label{factor} Let $S$ be an $\aleph_0$-categorical semigroup, let $\underline{X}\in S^t$ for some $t\in\N^0$,  and $\rho$ an $\underline{X}$-r.a.p. congruence on $S$. Then  $S/\rho$ is $\aleph_0$-categorical. 
\end{proposition} 

\begin{proof} Suppose $\underline{X}\in S^t$ and let $\underline{a}=(a_1 \rho, \dots , a_n \rho)$ and $\underline{b}=(b_1 \rho, \dots , b_n \rho)$  be a pair of $n$-tuples  of $S/\rho$ such that $(a_1,\dots,a_n,\underline{X}) \sim_{S,n+t} (b_1,\dots,b_n,\underline{X})$ 
 via $\phi\in \text{Aut}(S)$, say. Then $\underline{X}\phi=\underline{X}$, so that $\rho$ is preserved under the automorphism $\phi$, and there is thus an automorphism $\psi$ of $S/\rho$ given by 
\[ [a]_{\rho} \psi = [a\phi]_{ \rho} \quad (a\rho \in S/\rho). 
\] 
Since $[a_k]_{ \rho} \psi = [a_k \phi]_{\rho}=[b_k ]_{\rho}$ for each $1\leq k \leq n$, we have $\underline{a} \sim_{S/\rho, n} \underline{b}$, and so
\[ |(S/ \rho)^n / \sim_{S/ \rho, n}| \leq |S^{n+t} / \sim_{S,n+t}|<\aleph_0.
\] 
as $S$ is $\aleph_0$-categorical. Hence $S/\rho$ is $\aleph_0$-categorical. 
\end{proof} 

If we drop the condition on Proposition \ref{factor} that the congruence is relatively automorphism preserving then the statement is no longer true. 
An example of an $\aleph_0$-categorical group with a non $\aleph_0$-categorical quotient group  is given by Rosenstein  \cite{Ros76}.


\begin{corollary}\label{f.g cong cat} Let $S$ be an $\aleph_0$-categorical semigroup.
\begin{enumerate}\item If $\rho$ is a congruence preserved  by automorphisms, then 
$S/\rho$ is  $\aleph_0$-categorical.
\item If $S$ is inverse, then $S/\sigma$ is  an $\aleph_0$-categorical group.
\item The semigroup $S/\eh^{\flat}$ is  $\aleph_0$-categorical, so that if $S$ is inverse, then
$S/\mu$ is  $\aleph_0$-categorical.
\item If $\rho$ is a finitely generated congruence, then $S/\rho$ is  $\aleph_0$-categorical. 
\item If $I$ is an $\underline{X}$-pivoted r.c. ideal of $S$ for some finite tuple $\underline{X}$ of elements of $S$, then  $S/I$ is $\aleph_0$-categorical. 
\end{enumerate}
\end{corollary} 

\begin{proof} (1)-(3) follow from Example~\ref{ex:congs} 
and Proposition~\ref{factor}. For (4) we let $\rho=\langle (u_1,v_1),\hdots,(u_r,v_r)\rangle$ be a finitely generated congruence on $S$
and let $\underline{X}=(u_1,v_1,\hdots, u_n,v_n)$.
It is easy to see from the explicit  description of  
$\rho$ (see  \cite[Proposition 1.5.9]{Howie94}) that $\rho$ is an $\underline{X}$-pivoted r.a.p. congruence with pivot $\underline{X}$.  The result is then immediate from Proposition \ref{factor}. 

For (5), suppose that $I$ is an $\underline{X}$-p.r.c. ideal of $S$.  Let $\phi$ be an automorphism of $S$ which fixes $\underline{X}$, so that $I\phi=I$ since $I$ is an $\underline{X}$-pivoted r.c ideal.
 Then, for any $a,b\in S$, we have
\[  a \, \rho_I \, b \Leftrightarrow [a=b \text{ or } a,b\in I] \Leftrightarrow [a\phi=b\phi \text{ or } a\phi,b\phi\in I] \Leftrightarrow  a\phi \, \rho_I \, b\phi,
\] 
thus showing that $\rho_I$ is an $\underline{X}$-r.a.p. congruence, and once more we call upon Proposition~\ref{factor}.
\end{proof} 


As we have seen in Corollaries \ref{excluded} and \ref{f.g cong cat}, the RNT is adept at dealing with a range of finiteness conditions. We end this section by studying a final finiteness condition:  equivalence relations on a semigroup with finite equivalence classes. 

Let $S$ be a semigroup and $\tau$ an equivalence relation on $S$. For each $n\geq 1$, define an equivalence relation $\#_{S,\tau,n}$ on $S^n$ by $(a_1,\dots,a_n) \, \#_{S,\tau,n} \, (b_1,\dots,b_n)$ if and only if there exists an automorphism $\phi$ of $S$ such that $(a_k\tau)\phi=b_k\tau$ for each $1\leq k \leq n$.
 
\begin{proposition}\label{finite classes} Let $S$ be a semigroup and  $\tau$ an equivalence on $S$ with each $\tau$-class being finite.
 Then $|S^n/\#_{S,\tau,n}|$ is finite for each $n\geq 1$ if and only if $S$ is $\aleph_0$-categorical and $A=\{|m\tau|:m\in S\}$ is finite.  
\end{proposition} 

\begin{proof} 
Suppose that $|S^n/\#_{S,\tau,n}|$ is finite for each $n\geq 1$. Let $Z=\{\underline{a}_i:i\in\mathbb{N}\}$ be an infinite set of $n$-tuples of $S$, where $\underline{a}_i=(a_{i1},\dots,a_{in})$. Since $|S^n/\#_{S,\tau,n}|$ is finite, there exists an infinite subset $\{\underline{a}_i:i\in I\}$ of $Z$ such that $\underline{a}_i \, \#_{S,\tau,n} \,  \underline{a}_j$ for each $i,j\in I$. In particular, for each $i\in I$ there exists an automorphism $\phi_i$ of $S$ with $({a}_{ik}\tau) \phi_i={a}_{1k}\tau$ for each $1 \leq k \leq n$. Hence $a_{ik}\phi_i\in a_{1k}\tau$ for each $1\leq k \leq n$, so that
\[ \underline{a}_i\phi_i\in \{(z_1,\dots,z_n):z_k\in a_{1k}\tau\}. 
\] 
Notice that set $\{(z_1,\dots,z_n):z_k\in a_{1k}\tau\}$ is finite since each $\tau$-class is finite. Consequently, there exists distinct $i,j\in I$ such that $\underline{a}_i\phi_i=\underline{a}_j\phi_j$, so that $\underline{a}_i\phi_i\phi_j^{-1}=\underline{a}_j$. Hence $\underline{a}_i $ and $\underline{a}_j$ are automorphically equivalent. It follows that $S$ contains no infinite set of distinct $n$-automorphism types, and is thus $\aleph_0$-categorical by the RNT. 
 Furthermore, by our usual argument we have that $|A|$ is bound by $|S / \#_{S,\tau,1}|$.
 
Conversely, suppose $S$ is $\aleph_0$-categorical and $A$ is finite. Let  $\underline{m}=(m_1,\dots,m_n)$ and $\underline{m}'=(m_1',\dots,m_n')$ be a pair of $n$-tuples of $S^n$, under the condition that $|m_k\tau|=|m_k'\tau|$ for each $k$. Since each entry of an $n$-tuple of $S^n$ has $|A|$ potential cardinalities for its $\tau$-class, it follows that this condition has $|A|^n$ choices. For each $1\leq k \leq n$, let $m_k\tau=\{a_{k1},\dots,a_{ks_k}\}$ and $m_{k}'\tau =\{b_{k1},\dots,b_{ks_k}\}$, and let $T(n)=s_1+s_2+\cdots + s_n$. Suppose further that 
\[ (a_{11},\dots,a_{1s_1},a_{21},\dots,a_{2s_2},\dots ,a_{ns_n}) \, \sim_{S,T(n)} \, (b_{11},\dots,b_{1s_1},b_{21},\dots,b_{2s_2},\dots,b_{ns_n}),
\]  via $\phi\in \text{\aut}(S)$, say. Note that this condition also has finitely many choices  as $|S^{T(n)}/ \sim_{S,T(n)}|$ is finite for each $n\geq 1$ by the RNT. Moreover, $(m_k\tau)\phi=m_k'\tau$ for each $k$, since $a_{kr}\phi=b_{kr}$ for each $1\leq r \leq s_k$. Hence $\underline{m} \, \#_{S,\tau,n} \, \underline{m}'$, and so $|S^n/\#_{S,\tau,n}|$ is finite by Lemma \ref{counting bits}. 
\end{proof}

\begin{corollary}\label{finite H use E} Let $S$ be a regular semigroup with each maximal subgroup being finite.
 Then $S$ is $\aleph_0$-categorical if and only if $|E(S)^n/\sim_{S,n}|$ is finite for each $n\geq 1$. 
\end{corollary} 

\begin{proof} If $S$ is $\aleph_0$-categorical, then 
\[ |E(S)^n/\sim_{S,n}|\leq |S^n/\sim_{S,n}|<\aleph_0
\] 
for each $n\geq 1$ by the RNT.

Conversely, suppose $|E(S)^n/\sim_{S,n}|$ is finite for each $n\geq 1$ and consider a pair of $n$-tuples of  $S$ given by $\underline{a}=({a_1},\dots,{a_n})$ and $\underline{b}=({b_1},\,\dots,{b_n})$.  Since $S$ is regular, there exists idempotents $e_i,f_i,\bar{e}_i,\bar{f}_i$ of $S$ with $e_i \, \mathcal{R} \, a_i \, \mathcal{L} \, f_i$ and $\bar{e}_i \, \mathcal{R} \, b_i \, \mathcal{L} \, \bar{f}_i$ for each $1\leq i \leq n$. Suppose further that 
\[ (e_1,f_1,e_2,f_2,\dots,e_n,f_n) \, \sim_{S,2n} \, (\bar{e}_1,\bar{f}_1,\bar{e}_2,\bar{f}_2,\dots,\bar{e}_n,\bar{f}_n), 
\] 
 via $\phi\in \text{Aut}(S)$, say. Then as $\mathcal{R}$ and $\mathcal{L}$ are automorphism preserving we have that $R_{e_i}\phi=R_{\bar{e}_i}$ and $L_{f_i}\phi=L_{\bar{f}_i}$ for each $i$, so that 
\[ H_{a_i}\phi=(R_{a_i}\cap L_{a_i})\phi = (R_{e_i}\cap L_{f_i})\phi=R_{e_i}\phi \cap L_{f_i}\phi = R_{\bar{e}_i}\cap L_{\bar{f}_i}=H_{b_i}. 
\] 
Hence $\underline{a} \, \#_{S,\mathcal{H},n} \, \underline{b}$, and we have thus shown that 
\[ |S^n/\#_{S,\mathcal{H},n}|\leq |E(S)^{2n}/ \sim_{S,2n}| <\aleph_0. 
\] 
Since each maximal subgroup of  $S$ is finite, every $\mathcal{H}$-class of $S$ is finite by \cite[Lemma 2.2.3]{Howie94} and the regularity of $S$. Hence $S$ is $\aleph_0$-categorical by Proposition \ref{finite classes}. 
\end{proof}

\section{Principal factors of an $\aleph_0$-categorical semigroup} \label{Sec princ} 

Our interest in this section is in determining how $\aleph_0$-categoricity effects the principal factors of a semigroup.
 Our main result is that the principal factors of an $\aleph_0$-categorical semigroup behave in much the same way as the maximal subgroups:

\begin{theorem}\label{principal fac} The principal factors of an $\aleph_0$-categorical semigroup $S$ are $\aleph_0$-categorical, and either completely 0-simple, completely simple or null. Moreover, $S$ has only finitely many principal factors, up to isomorphism.  
\end{theorem} 

\begin{proof} For each $a\in S$ let $J(a)=S^1aS^1$ and $I(a)=J(a)\setminus J_a$. Since $S$ is $\aleph_0$-categorical, the ideals $J(a)$ are $\aleph_0$-categorical by Corollary~ \ref{J}. Let $\phi$ be an automorphism of $S$ such that $a\phi=b$. Then $J(a)\phi=J(b)$ as $\{(J(a),a):a\in S\}$ is a system of 1-pivoted p.r.c. subsemigroups of $S$. Moreover,  as $\mathcal{J}$ is preserved under automorphisms we have $J_a\phi=J_b$, and so 
\[ I(a)\phi = (J(a)\setminus J_a)\phi=J(b)\setminus J_b = I(b).
\]
 Consequently, $\{(I(a),a):a\in S\}$ is a system of 1-pivoted p.r.c. subsemigroups of $S$ and, in particular, $I(a)$ is an $a$-pivoted p.r.c. ideal of $J(a)$ for each $a\in S$. Hence $J(a)/I(a)$ is $\aleph_0$-categorical by Corollary~\ref{f.g cong cat}.
 If the kernel $K(S)$ of $S$ exists, that is, the unique minimum ideal of $S$, then it is a $\mathcal{J}$-class of $S$, and is thus $\aleph_0$-categorical. Hence each principal factor of $S$ is $\aleph_0$-categorical. 
 
Moreover, as  $\phi|_{J(a)}$ is an isomorphism from $J(a)$ to $J(b)$ with $I(a)\phi|_{J(a)}=I(b)$, it follows that the isomorphism $\phi|_{J(a)}$ preserves $\rho_{I(a)}$ and $\rho_{I(b)}$, and so $\phi$ induces an isomorphism from $J(a)/I(a)$  to $J(b)/I(b)$. Hence the set $\{J(a)/I(a):a\in S\}$ of non kernel principal factors of $S$ has at most $|S/\sim_{S,1}|$ elements, up to isomorphism. Since $K(S)$ is unique, if it exists, $S$ has only finitely many principal factors, up to isomorphism. 

By \cite[Lemma 2.39]{Clif&Pres61}, the principal factors of $S$ are either 0-simple, simple or null. A periodic (0)-simple semigroup is completely (0-)semigroup (the result for 0-simple semigroups is given in \cite[Corollary 2.56]{Clif&Pres61}, from which the simple case follows). Hence as an $\aleph_0$-categorical semigroup is periodic by Corollary
~\ref{cor:indexperiod}, each principal factor is either completely 0-simple, completely simple or null.
\end{proof} 

Recall that every null semigroup  is $\aleph_0$-categorical  by Example \ref{null cat}. To understand the $\aleph_0$-categoricity of an arbitrary semigroup it is therefore essential to examine the completely simple and completely 0-simple cases. The $\aleph_0$-categoricity of an arbitrary completely (0-)simple semigroup will be the main topic of our subsequent paper. For now, we restrict our attension to the $\aleph_0$-categoricity of inverse completely 0-simple semigroups. 

  The \textit{Brandt semigroup $S$ over a  group $G$ with index set $I$}, denoted $B^0[G;I]$, is the set $(I\times G \times I)\cup \{0\}$ with multiplication $(i,g,j)0=0(i,g,j)=00=0$ and
\[ (i,g,j)(k,h,l) = \left\{
\begin{array}{ll}
(i,gh,l) & \text{if       } j=k,\\
0 & \text{if }  j \neq k.
\end{array}
\right. 
\]
Every Brandt semigroup is an inverse completely 0-simple semigroup and, conversely, an inverse completely 0-simple semigroup is isomorphic to some Brandt semigroup
\cite[Theorem 5.1.8]{Howie94}. Our early interest in Brandt semigroups from an $\aleph_0$-categorical perspective is due to the simplicity with which their automorphisms may be determined. 
The automorphism theorem for Brandt semigroups below is a direct consequence of  \cite[Theorem 3.4.1]{Howie94}:     

\begin{theorem} \label{Brandt iso}  Let $S=B^0[G;I]$  be a Brandt semigroup. Let $\theta$ be an automorphism of $G$, and $\pi$ a bijection of $I$. 
Then the map $\psi:S\rightarrow S$  given by $0\psi=0$ and $(i,g,j)\psi=(i\pi,g\theta,j\pi)$ for each $(i,g,j)\in S\setminus \{0\}$ is an automorphism, denoted $\psi=(\theta;\pi)$. 
Conversely, every autormorphism of $B^0[G;I]$ may be constructed in this manner.
\end{theorem}  

\begin{proposition} \label{brandt cat} A Brandt semigroup $S=B^0[G;I]$ is $\aleph_0$-categorical if and only if $G$ is $\aleph_0$-categorical. 
\end{proposition} 

\begin{proof}  ($\Rightarrow$) Since $G$ is isomorphic to each non-zero maximal subgroup $B_i=\{(i,g,i):g\in G\}$ of $S$, the result follows from Corollary \ref{cor:green}.

($\Leftarrow$) By the RNT and Corollary \ref{excluded}, to prove the $\aleph_0$-categoricity of $S$ it suffices  to show that the number of $n$-automorphism types of $S^*=S \setminus \{0\}$ is finite for each $n\geq 1$.
 Let $\underline{a}=((i_1,g_1,j_1),\dots, (i_n,g_n,j_n))$ and $\underline{b}=((k_1,h_1,\ell_1),\dots,(k_n,h_n,\ell_n))$ be a pair of $n$-tuples of $S^*$ such that 
\begin{enumerate} 
\item $(i_1,\dots, i_n,j_1,\dots, j_n) \, \natural_{I,2n} \, (k_1,\dots,k_n,\ell_1,\dots,\ell_n)$,
\item   $(g_1,\dots,g_n) \sim_{G,n} (h_1,\dots,h_n)$, via $\theta\in \text{Aut}(G)$, say.
\end{enumerate} 
By condition (1) there exists a bijection $\pi$ from $\{i_1,\dots,i_n,j_1,\dots,j_n\}$  to $\{k_1,\dots,k_n,\ell_1,\dots,\ell_n\}$ given by $i_r\pi=k_r$ and $j_r\pi=\ell_r$ ($1\leq r \leq n$). Moreover, condition (1) has $B_{2n}$ choices, which is finite. By the $\aleph_0$-categoricity of $G$, condition (2)  also has finitely many choices. Take a bijection $\bar{\pi}$ of $I$ which extends $\pi$. Then $\psi=(\theta;\bar{\pi})$ is an automorphism of $S$ by Theorem \ref{Brandt iso}, and is such that 
\[ (i_r,g_r,j_r)\psi=(i_r\bar{\pi},g_r\theta,j_r\bar{\pi})=(i_r\pi, h_r,j_r\pi)=(k_r,h_r,\ell_r)
\] 
for each $1\leq r \leq n$. Hence $\underline{a} \, \sim_{S,n} \, \underline{b}$, and so $S$ is $\aleph_0$-categorical by Lemma \ref{counting bits}. 
\end{proof}

The classification of $\aleph_0$-categorical Brandt semigroups is an example of building $\aleph_0$-categorical semigroups from $\aleph_0$-categorical `ingredients', in this case  groups (and sets). The rest of the article is attributed to investigating  constructions of this form for a number of rudimentary examples, including 0-direct union and semidirect products. 

\section{Building $\aleph_0$-categorical semigroups: direct sums and 0-direct union} \label{Sec 0-direct} 

Let $I$ be an   indexing set and suppose that for each $i\in I$ we have a  monoid $M_i$ with identity 
$1_i$. By the
{\em direct sum} $S=\bigoplus_{i\in I}M_i$ of the monoids $M_i,i\in I$ we mean the submonoid
\[\{ (m_i)_{i\in I}: m_i=1_i\mbox{ for all but finitely many }i\in I\}\]
of the direct product $P=\Pi_{i\in I}M_i$. 
Rosenstein~\cite{Rosenstein73} showed that any group that is a direct sum of copies of finitely many finite groups is  $\aleph_0$-categorical if and only if every group which occurs infinitely often in the sum is abelian. For our purposes we extend his result slightly as    follows.

\begin{lemma}\label{lem:dpgps} (cf. \cite[Theorem 3]{Rosenstein73}) Let $S=\bigoplus_{i\in \N}M_i$ where each $M_i$ is a finite group. Then if $S$ is $\aleph_0$-categorical, all but finitely many of the  $M_i$'s are abelian.  
\end{lemma}
\begin{proof} Suppose for contradiction that 
\[M_{i_1}, M_{i_2}, \hdots\]
are all non-abelian. Pick $a_{j}\in M_{i_j}$ with 
$a_{j}\notin Z(M_{i_j})$, so that  $|\{b^{-1}a_jb:b\in M_{i_j}\}|=m_j>1$. 
Let $(\underline{s_j})\in S$ be defined by
\[(\underline{s_j})_{i_k}=a_k\mbox{ for }1\leq k\leq j, (\underline{s_j})_i=1_i\mbox{ else}.\]
Then 
$|\{\underline{b}^{-1}\underline{s_j}\underline{b}:\underline{b}\in S\}|=m_1\hdots m_j$ so that for $i\neq j$ we cannot have  $(\underline{s_i})\sim_{S,1}(\underline{s_j})$,  contradicting the $\aleph_0$-categoricity of $S$. 
\end{proof}

\begin{theorem}\label{mon:ds}  Let $S=\bigoplus_{i\in \N}M_i$ be a direct sum of finite
monoids $M_i$.  Then $S$ is   $\aleph_0$-categorical if and only if 
$S$ is a direct product of a finite monoid and an abelian group of bounded order. 
\end{theorem}
\begin{proof} We first show that all but  finitely many of the monoids $M_i$ are groups. 
Suppose we have an infinite sequence
\[M_{i_1}, M_{i_2}, \hdots\]
such that each $M_{i_j}$ is not  a group. By our hypothesis, for each $j\in\N$ we may choose a non-identity idempotent
$e_j\in M_{i_j}$ such that $e_j$ is maximal in $E(M_{i_j})\setminus\{ 1_{i_j}\}$. 

Consider the sequence
\[\underline{s_1}, \underline{s_2},\underline{s_3},\hdots\]
where
\[(\underline{s_j})_{i_k}=e_k\mbox{ for }1\leq k\leq j, (\underline{s_j})_i=1_i\mbox{ else}.\]
For each $j\in\N$ there are exactly $2^j$ idempotents of $S$ above the  element $\underline{s_j}$ in the natural partial order, so that the elements 
$\underline{s_j}$ lie in distinct $\sim_{S,1}$-classes,  a contradiction. 
Thus  $S=M\times G$ where $M$ is a  finite monoid  and
$G$ is a direct sum of finite groups. 

The group of units of $S$ is $\aleph_0$-categorical, and is a direct sum of the group of units $H$ of $M$ and $G$. 
By Lemma~\ref{lem:dpgps} all but finitely many of the constituents of the direct sum forming $G$  are abelian, so that
$G=K\times W$ where $K$ is finite and $W$ is an abelian group of bounded order, hence
$\aleph_0$-categorical by \cite[Theorem 2]{Rosenstein73}. Thus $S=M\times K\times W$ where
$M\times K$ is finite and $W$ is an abelian group of bounded order.

The converse is clear as $\aleph_0$-categoricity is preserved by finite direct product.
\end{proof} 

To translate to the semigroup case requires some care, as here the direct sum does not embed into the direct product.  Let  $M_i$ be a semigroup for each $i\in I$. By the
{\em direct sum} of the semigroups $M_i$ ($i\in I$) we mean the semigroup
\[ S=\langle \bigcup_{i\in I} \overline{M_i} : \overline{m_i}\, \overline{m_i'}=\overline{m_im'_i},\, \overline{m_i}\, \overline{m_j}=\overline{m_j}\, \overline{m_i}\, \forall i\neq j, m_i,m_i'\in M_i,m_j\in M_j\rangle, \]
where $\overline{M_i}=\{ \overline{m_i}:m_i\in M_i\}$ for all $i\in I$.

\begin{proposition}\label{thm:dssemigroups} Let $S$ be the
{direct sum} of the finite semigroups $M_i$ ($i\in I$). Then $S$ is $\aleph_0$-categorical 
if and only if $I$ (and hence $S$) is finite. 
\end{proposition}
\begin{proof}  Suppose that $S$ is $\aleph_0$-categorical. For each
$M_i$ we choose a maximal idempotent $e_i$. If $I$ is infinite, then
without loss of generality we may take $I=\N$. Let
$\overline{s_i}=\overline{e_1}\overline{e_2}\hdots \overline{e_n}$. Notice that for each $\overline{s_i}$ there are precisely $2^i-1$ 
idempotents greater than $\overline{s_i}$, so that each $\overline{s_i}$ lies in a distinct $\sim_{S,1}$-class.
Thus $I$ is finite. The converse is immediate by Corollary \ref{cor:finite}.
\end{proof}

Given the disappointing nature of Proposition~\ref{thm:dssemigroups} we focus attention on a different construction, which yields useful results. 
The basic definitions and results are taken from \cite{Ciric95}. 

A semigroup with zero $S$ is a \textit{0-direct union} or \textit{orthogonal sum} of the subsemigroups $S_i$ ($i\in A$) with zero, if the following hold: 
 \begin{enumerate}[label=(\arabic*)]
\item  $S_i\neq \{0\}$ for each $i\in A$; 
 \item  $S=\bigcup_{i\in A} S_i$; 
 \item $S_i\cap S_j=S_iS_j=\{0\}$ for each $i\neq j$.
  \end{enumerate}
 We denote $S$ as $\bigsqcup_{i\in A}^0 S_i$. The family  $\mathcal{S}=\{S_i:i\in A\}$ is called a \textit{0-direct decomposition} of $S$, and the $S_i$ are called the \textit{summands} of $S$. Note that each summand of $S$ forms an ideal of $S$.  If $\mathcal{S}$ and $\mathcal{S}'$ are a pair of 0-direct decompositions of $S$, then we say that $\mathcal{S}$ is \textit{greater than} $\mathcal{S}'$ if each member of $\mathcal{S}$ is a subsemigroup of some member of $\mathcal{S}'$. We say that $S$ is \textit{0-directly indecomposable} if $\{S\}$ is the unique 0-direct decomposition of $S$.

\begin{example} Let $B^0[G;I]$ be a Brandt semigroup, and consider the group with zero $B^0_i=\{(i,g,i):g\in G\}\cup\{0\}$ for each $i\in I$. Then $B_iB_j=\{0\}$ if $i\neq j$, and so $\bigcup_{i\in I}B_i^0$  forms a 0-direct union of the subsemigroups $B_i^0$.
\end{example} 

A subset $T$ of a semigroup $S$ is \textit{consistent} if, for $x,y\in S$, $xy\in T$ implies that $x,y\in T$. A subset $T$ of a semigroup with zero is \textit{0-consistent} if $T\setminus \{0\}$ is consistent. The integral connection between 0-consistency and 0-direct decompositions is that a semigroup with zero $S$ is 0-directly indecomposable if and only if $S$ has no proper 0-consistent ideals \cite[Lemma 4]{Ciric95}. Consequently, every completely 0-simple semigroup is 0-directly indecomposable. 

 The central result of  \cite{Ciric95} was proving that that every semigroup with zero has a greatest 0-direct decomposition, and that the summands of such a decomposition are precisely the 0-directly indecomposable ideals. The importance of the existence of a greatest 0-direct decomposition for $\aleph_0$-categoricity is highlighted in the following proposition.  

 \begin{proposition}\label{iso 0-direct} Let $S$ be a semigroup with zero and let $\mathcal{S}= \{S_i:i\in A\}$ be the greatest 0-direct decomposition of $S$. Let $\pi:A\rightarrow A$ be a bijection and $\phi_i:S_i\rightarrow S_{i\pi}$ an isomorphism for each $i \in A$. Then the map $\phi:S \rightarrow S$ given by 
 \[ s_i\phi = s_i\phi_i \quad  (s_i\in S_i)
 \] 
 is an automorphism of $S$, denoted $\phi=\bigsqcup^0_{i\in A} \phi_i$. Moreover, every automorphism of $S$ can be constructed in this way. 
 \end{proposition} 
 
 \begin{proof}
 Let $\phi$ be constructed as in the hypothesis of the proposition. Since $0\phi_i=0$ for each $i\in A$ the map is well-defined, and it is clearly bijective. Let $a\in S_i$ and $b\in S_j$. If $i=j$ then 
 \[ (ab)\phi = (ab)\phi_i=(a\phi_i)(b\phi_i)=(a\phi)(b\phi),
 \] 
 and if $i\neq j$ then 
 \[ (ab)\phi = 0\phi = 0 = (a\phi_i)(b\phi_j) =(a\phi)(b\phi).
 \] 
 Hence $\phi$ is an isomorphism. 
 
 Conversely, if $\phi'$ is an automorphism of $S$, then 
 \[ \mathcal{S}\phi'=\{S_i\phi':i\in A\} 
 \] 
 is clearly a 0-direct decomposition of $S$. For each summand $S_i$ there exists $k\in A$ such that $S_i\subseteq S_k\phi'$ since $\mathcal{S}$ is the greatest 0-direct decomposition. If $S_i\subseteq S_k\phi'\cap S_{k'}\phi'$ then $S_i=\{0\}$ as $\mathcal{S}\phi'$ is a 0-direct decomposition of $S$, a contradiction. Hence the element $k$ is unique. On the other hand, if  $S_i,S_j\subseteq S_k\phi'$, then $S_i\phi'^{-1},S_j\phi'^{-1}\subseteq S_k$, and so as $\{S_i\phi'^{-1}:i\in A\}$ is also a 0-direct decomposition of $S$, we have that $i=j$ since $S_k$ is 0-direct indecomposable. Hence there exists a bijection $\pi'$ of $A$ such that $S_{i}\phi'=S_{i\pi'}$ for each $i\in A$ as required.  
 \end{proof}

 \begin{proposition} \label{cat 0-direct} Let $S$ be a semigroup with zero and let $\mathcal{S}= \{S_i:i\in A\}$ be the greatest 0-direct decomposition of $S$. Then $S$ is $\aleph_0$-categorical if and only if each $S_i$ is $\aleph_0$-categorical and $\mathcal{S}$ is finite, up to isomorphism. 
 \end{proposition}  
  
 \begin{proof}  It follows immediately from Proposition \ref{iso 0-direct} that $\{(S_i,x_i):i\in A\}$ forms a system of 1-pivoted p.r.c subsemigroups of $S$ for any $x_i\in S_i^*$. Hence if $S$ is $\aleph_0$-categorical then each $S_i$ is $\aleph_0$-categorical and $\mathcal{S}$ is finite, up to isomorphism, by Proposition  \ref{pairwise-rel-char}. 
 
 Conversely, let each summand be $\aleph_0$-categorical and suppose there exists exactly $r\in \mathbb{N}$ summands, up to isomorphism. Let $S_{\rho_1},\dots, S_{\rho_r}$ be representatives of the isomorphism types of the summands of $S$ and, for each $\mu\in A$, let $\phi_{\mu}$ be an isomorphism from $S_\mu$ to its unique isomorphic representative
 in  $S_{\rho_1},\dots, S_{\rho_r}$.  By Corollary \ref{excluded} it suffices to  show that the number of $n$-automorphism types of $S^*=S\setminus \{0\}$ is finite  for each $n\geq 1$.  Let $(a_1,\dots,a_n)$ and $(b_1,\dots,b_n)$ be a pair of $n$-tuples of $S^*$ with $a_k\in S_{i_k}$ and $b_k\in S_{j_k}$ for each $1\leq k \leq n$, say.  Impose the condition that $a_i,a_j$ belong to the same summand if and only if $b_i,b_j$ belong to the same summand,  for each $1\leq i, j\leq n$. This is clearly equivalent to the map $\pi:\{i_1,\dots,i_n\} \rightarrow \{j_1,\dots,j_n\}$ given by $i_k\pi=j_k$ being a bijection, and thus the number of choices for this condition is equal to $B_n$. Suppose also that $S_{i_k}\cong S_{j_k}$ for each $k$, noting that this condition has $r^n$ choices. For each $1\leq s \leq r$, let  $a_{s1},\dots,a_{sn_s}$ be precisely the entries of $\underline{a}$ which are elements of summands isomorphic to $S_{\rho_s}$, noting that the same is true of $b_{s1},\dots,b_{sn_s}$ by our second condition. Note also that $\{1,\dots,n\}=\{11,\dots,1n_1,21,\dots,2n_2,\dots,rn_r\}$. 
We impose a final condition on our pair of $n$-tuples which forces, for each $1\leq s \leq r$, 
\[ (a_{s1}\phi_{i_{s1}},\dots,a_{sn_s}\phi_{i_{sn_s}}) \, \sim_{S_{\rho_s},n_s} \, (b_{s1}\phi_{j_{s1}},\dots,b_{sn_s}\phi_{j_{sn_s}})
\] 
via $\psi_s\in \text{Aut}(S_{\rho_s})$, say (where if $n_s=0$ then we  take any automorphism of $S_{\rho_s}$). By the $\aleph_0$-categoricity of each $S_{\rho_s}$ this condition also has finitely many choices. For each $1 \leq s \leq r$ and $1\leq t \leq n_s$ we have that $\phi_{i_{st}}\psi_s \phi_{j_{st}}^{-1}$ is an isomorphism from $S_{i_{st}}$ to $S_{j_{st}}$, and is such that 
\[ a_{st}\phi_{i_{st}}\psi_s \phi_{j_{st}}^{-1} = b_{st}.
\]
Let $\bar{\pi}$ be a bijection of $A$ which extends $\pi$ and which preserves the isomorphism types of the summands, so that $S_i\cong S_{i\bar{\pi}}$. Such a bijection exists since each $S_{i_k}$ is isomorphic to $S_{j_k}$. For each $i\in A\setminus \{i_1,\dots, i_n\}$, let $\Psi_i$ be an isomorphism from $S_i$ to $S_{i\bar{\pi}}$, and we let $\Psi_{i_{st}}=\phi_{i_{st}}\psi_s \phi_{j_{st}}^{-1}$ for each $1 \leq s \leq r$ and $1\leq t \leq n_s$. Then $\Psi=\bigsqcup^0_{i\in A} \Psi_i$ is an automorphism of $S$ by Proposition \ref{iso 0-direct}, and is such that $\underline{a}\Psi=\underline{b}$ since $\Psi$ extends each $\phi_{i_{st}}\psi_s \phi_{j_{st}}^{-1}$. Since each of our conditions has only finitely many choices, $(S^*)^n$ has only finitely many $n$-automorphism tuples (over $S$) by Lemma \ref{counting bits}, and thus $S$ is $\aleph_0$-categorical.
 \end{proof}
 
When studying $\aleph_0$-categorical semigroups with zero, it therefore suffices to examine 0-directly indecomposable semigroups. 

We observe that without the condition of $\mathcal{S}$ being the greatest 0-direct decomposition of $S$, the converse of Proposition \ref{cat 0-direct} need not be true. For example, for each $n\geq 1$, let $N_n$ be a null semigroup on $n$ non-zero elements. Then $N=\bigsqcup^0_{i\in\mathbb{N}} N_i$ is a countably infinite null semigroup, and is thus $\aleph_0$-categorical by Example \ref{null cat}. However the set of summands of $N$ is not finite, up to isomorphism. 
 
 A semigroup $S$ with zero is called \textit{primitive} if each of its non-zero idempotents is primitive. It  follows from the work of Hall in \cite{Hall70} that a regular semigroup $S$ is primitive if and only if $S$ is isomorphic to a 0-direct union of completely 0-simple semigroups.
 Since each completely 0-simple semigroup is 0-directly indecomposable, we obtain the following immediate consequence to Proposition \ref{cat 0-direct}.  
 
 \begin{corollary}\label{0-direct rees} Let $S_i$ ($i\in A$)  be a collection of completely 0-simple semigroups. Then $\bigsqcup_{i\in A}^0 S_i$ is $\aleph_0$-categorical if and only if each $S_i$ is $\aleph_0$-categorical and $\{S_i:i\in A\}$ is finite, up to isomorphism. 
 \end{corollary} 
 
 A classification of primitive regular $\aleph_0$-categorical semigroups via its completely 0-simple semigroup ideals then follows. In particular, by Theorem \ref{Brandt iso} and Proposition \ref{brandt cat} we  have the following classification of primitive inverse semigroups: 
 
\begin{corollary} A primitive inverse semigroup $S$ is $\aleph_0$-categorical if and only if $S$ is isomorphic to  $\bigsqcup_{i\in A}^0 \mathcal{B}^0[G_i;I_i]$ such that each $G_i$ is $\aleph_0$-categorical,  $\{G_i:i\in A\}$ is  finite up to isomorphism, and $\{|I_i|:i\in A\}$ is finite.    
\end{corollary} 

\section{Building $\aleph_0$-categorical semigroups: Semidirect products and McAlister's P-Theorem}\label{sec:last}

Given that $\aleph_0$-categoricity has been shown by Grzegorczyk to be inherited by finite direct product \cite{Grzeg}, the next natural question is to assess semidirect products. In this section we do so in the case of a semigroup acting on a finite semigroup.

 Our work requires the following variant of $\aleph_0$-categoricity, and the subsequent pair of lemmas: 

\begin{definition}\label{def:aug} Given a semigroup $S$ and a collection $\mathcal{A}=\{S_i:i\in A\}$ of subsets of $S$, we let Aut$(S;\mathcal{A})$ denote the group of automorphisms of $S$ which fix each $S_i$ ($i\in A$) setwise. We call $S$ \textit{$\aleph_0$-categorical over $\mathcal{A}$} if Aut$(S;\mathcal{A})$ has finitely many orbits on its action on $S^n$ for each $n\geq 1$. We let $\sim_{S,\mathcal{A},n}$ denote the resulting equivalence relation on $S^n$. 
\end{definition} 

With notation as above, Definition~\ref{def:aug} is equivalent to the structure
consisting of the semigroup $S$ together with a collection of unary relations corresponding to the subsets $S_i$ ($i\in I$), being $\aleph_0$-categorical. Moreover, if $\underline{X}=(x_1,\dots,x_r)$ is a tuple of elements $S$, then the condition that  $S$ is $\aleph_0$-categorical over $\underline{X}$ is equivalent to $S$ being $\aleph_0$-categorical over $\{\{x_1\},\dots,\{x_r\}\}$. 

\begin{lemma}\label{cat over set} Let $S$ be a semigroup with a system of $t$-pivoted p.r.c. subsets $\{(S_i,\underline{X}_i):i\in I\}$. Then $S$ is $\aleph_0$-categorical over $\mathcal{A}=\{S_i:i\in I\}$ if and only if $S$ is $\aleph_0$-categorical and $\mathcal{A}$ is finite. 
\end{lemma} 

\begin{proof}  If $S$ is $\aleph_0$-categorical over $\mathcal{A}$, then trivially $S$ is $\aleph_0$-categorical.  Suppose $i,j\in I$ are such that $\underline{X}_i \, \sim_{S,\mathcal{A},r} \, \underline{X}_j$, via $\phi\in \text{Aut}(S;\mathcal{A})$, say. Then $S_i\phi=S_i$, while $S_i\phi=S_j$ since $\{(S_i,\underline{X}_i):i\in I\}$ is a system of $t$-pivoted p.r.c. subsets of $S$. Hence $S_i=S_j$, and so the cardinality of $\mathcal{A}$ is bound by the number of $r$-automorphism types over $\mathcal{A}$.  

Conversely, suppose $S$ is $\aleph_0$-categorical with $\mathcal{A}$ finite, say $\mathcal{A}=\{S_1,\dots, S_r\}$. Let $\underline{a}=(a_1,\dots,a_n)$ and $\underline{b}=(b_1,\dots,b_n)$ be a pair of $n$-tuples of $S$ such that $(\underline{a},\underline{X}_1,\dots,\underline{X}_r) \, \sim_{S,n+rt} \, (\underline{b},\underline{X}_1,\dots,\underline{X}_r)$, via $\psi\in \text{Aut}(S)$, say. Then as each pivot is fixed by $\psi$, the sets $S_i$ are setwised fixed by $\psi$, so that $\psi\in \text{Aut}(S;\mathcal{A})$. Hence as $\underline{a}\psi=\underline{b}$ we have that
\[ |S^n/\sim_{S,\mathcal{A},n}|\leq |S^{n+rt}/\sim_{S,n+rt}|<\aleph_0
\]
and so $S$ is $\aleph_0$-categorical over $\mathcal{A}$. 
\end{proof}

A simple adaptation of the proof of the lemma above also gives: 

\begin{lemma}\label{over char} Let $S$ be a semigroup,
let $t,r\in\N$, and for each $k\in \{ 1,\hdots, r\}$ let $\underline{X}_k\in S^t$. Suppose also that   $S_k$ is an $\underline{X}_k$-pivoted relatively characteristic subset of $S$ for $1\leq k \leq r$.  Then $S$ is $\aleph_0$-categorical if and only if $S$ is $\aleph_0$-categorical over $\{S_1,\dots,S_r\}$.  
\end{lemma}

Now suppose $S$ is  a semigroup acted on (on the left) by a monoid $T$ via  by endomorphisms. That is,
we have a map $T\times S\rightarrow S$ denoted by $(t,s)\mapsto t\cdot s$, such that for all $t,t'\in T$ and $s,s'\in S$ we have $tt'\cdot s=t\cdot (t'\cdot s)$, $1\cdot s=s$ and $t\cdot (ss')=(t\cdot s)(t\cdot s')$. We may then construct a semigroup on the set $S\times T$ with binary operation $(s,t)(s',t')=(s(t\cdot s'),tt')$. The resulting semigroup is denoted by $S\rtimes T$, and is called a \textit{semidirect product} of $S$ by $T$. 

Given a semidirect product $S\rtimes T$, we define a relation $\kappa$ on $T$ by 
\begin{equation} \label{action rel} t \, \kappa \, t' \Leftrightarrow s(t\cdot s') = s(t' \cdot s') \,\,  (\forall s,s'\in S). 
\end{equation} 
Then $\kappa$ is clearly an equivalence relation on $T$, and if $S$ is finite then $T/\kappa$ is finite.  If $S$ is a monoid and $t\cdot 1_S=1_S$ for all $t\in T$, then we say that $T$ acts
{\em monoidally}; note in this case $S\rtimes T$ is a monoid, and the definition of $\kappa$ simplifies to 
\[ t \, \kappa \, t' \Leftrightarrow t\cdot s' = t' \cdot s' \,\,  (\forall s'\in S).\]

\begin{proposition} \label{semidirect cat} 
Let $M=S\rtimes T$ be a semidirect product of $S$ and $T$, where $S$ is finite.  
If $T$ is $\aleph_0$-categorical over $T/\kappa$, then  $M$ is $\aleph_0$-categorical.

The converse holds if $S$ is a monoid with trivial group of units and $T$ acts monoidally, or if $S$ is a semilattice.
\end{proposition} 
\begin{proof} Suppose that 
$T$ is $\aleph_0$-categorical over $T/\kappa$. 
Let $\underline{a}=((s_1,t_1),\dots,(s_n,t_n))$ and $\underline{b}=((s_1',t_1'),\dots,(s_n',t_n'))$ be $n$-tuples of $M$ under the conditions that
\begin{enumerate}
\item $s_k=s_k'$ for each $1\leq k \leq n$, 
\item $(t_1,\dots,t_n) \, \sim_{T,T/\kappa,n} \, (t_1',\dots,t_n')$ via $\theta\in \text{Aut}(T;T/\kappa)$, say. 
\end{enumerate}
Note that the first condition has $|S|^n$ choices. We claim that the bijection $\phi:M\rightarrow M$ given by $(s,t)\phi=(s,t\theta)$ is an automorphism of $M$. Given $(s,t),(s',t')\in M$,
\begin{align*}
((s,t)(s',t'))\phi & = (s(t\cdot s'),tt')\phi = (s(t \cdot s'),(tt')\theta)= (s(t\theta \cdot s'),(tt')\theta)\\
& = (s(t\theta \cdot s'),(t\theta)(t'\theta)) = (s,t\theta)(s',t\theta) = (s,t)\phi (s',t')\phi,
\end{align*}
where the third equality is due to $t \, \kappa \, t\theta$, so in particular $s(t \cdot s') = s(t\theta \cdot s')$. Hence $\phi$ is indeed an automorphism of $M$. Moreover, for each $1 \leq k \leq n$, 
\[ (s_k,t_k)\phi = (s_k,t_k\theta) = (s_k',t_k')\]
and so $\underline{a} \phi = \underline{b}$. We therefore have that 
\[ |M^n/\sim_{M,n}| \leq |S|^n \cdot | T^n/\sim_{T,T/\kappa,n}| < \aleph_0
\] 
since $S$ is finite and $T$ is $\aleph_0$-categorical over $T/\kappa$. Hence $M$ is $\aleph_0$-categorical by the RNT. 

Conversely, suppose that $M$ is \ale. Enumerate the elements of $S$ as
$\{ s_1,\hdots ,s_r\}$. Let $s=1$ if $S$ is a monoid and let $s=0$ where $0$ is the least idempotent of $S$ if $S$ is a semilattice. Let
$\underline{t}=(t_1,\hdots ,t_n)$ and $\underline{u}=(u_1,\hdots, u_n)$ be $n$-tuples of
$T$ under the conditions that 
\begin{enumerate}
\item $t_k \, \kappa\,  u_k$ for each $1\leq k \leq n$, 
\item $\big((s_1,1),\hdots ,(s_r,1), (s,t_1), \hdots ,(s,t_n)\big)  \, \sim_{M,n}
\, \big((s_1,1),\hdots ,(s_r,1), (s,u_1), \hdots ,(s,u_n)\big)$ via $\theta\in \Aut M$, say. 
\end{enumerate}

For any $t\in T$ we define $t\phi$ by $(s,t)\theta=(s',t\phi)$. We claim that $\phi\in \Aut T$ and preserves $T/\kappa$.

\bigskip {\em Case (i): $S$ a monoid with trivial group of units, so $s=1_S$}. We first show that $(1_S,t)\theta=(1_S,t\phi)$. By definition, we have $(1_S,t)\theta=(s',t\phi)$. Choose $(b,w)\in S\rtimes T$ such that $(b,w)\theta=(1_S, t\phi)$. Then
\[(1_S,t)\theta=(s',t\phi)=(s',1)(1_S, t\phi)=(s',1)\theta (b,w)\theta=
((s',1)(b,w))\theta=(s'b,w)\theta\]
so that $1_S=s'b$, giving $s'=1_S$ as $H_1$ is trivial and $S$ is finite (giving that an element with a left inverse lies in $H_1$). 

We have thus shown that $(1_S,t)\theta=(1_S, t\phi)$, whence it follows that for any $u\in S, t\in T$ we have
\[(u,t)\theta=\big((u,1)(1_S,t)\big)\theta=(u,1)(1_S,t\phi)=(u,t\phi).\]
It is now easy to see that $\phi\in \Aut T$, since $T$ acts monoidally. 

\bigskip {\em Case (ii): $S$ a semilattice, so $s=0$}.  We first show that $(0,t)\theta=(0,t\phi)$. 
To see this, notice that, making use of (2),
\[(0',t\phi)=(0,t)\theta=\big((0,1)(0,t)\big)\theta=
(0,1)(0',t\phi)=(0,t\phi)\]
giving $0'=0$. It is now easy to see that $\phi$ yields an automorphism of $T$.

Let $e\in S$ and $t\in T$ and suppose that $(e,t)\theta=(e',t')$. 
  Then
\[(0,t')=(0,1)(e',t')=(0,1)\theta (e,t)\theta
=\big((0,1)(e,t)\big)\theta=(0,t)\theta=(0,t\phi),\] 
so that $t'=t\phi$. Let $u\in S$ with $u>0$. We may then suppose for induction that for all $v\in S$ with $u>v$ and for all $t\in T$ we have $(v,t)\theta=(v,t\phi)$. Then
with $(u,t)\theta=(u',t\phi)$ we have
\[(u',t\phi)=(u,t)\theta=\big((u,1)(u,t)\big)\theta=(u,1)(u',t\phi)=(uu',t\phi),\]
so that $u'=uu'$ and $u'\leq u$. If $u'<u$ we are led to the contradiction that
$(u',t)\theta=(u',t\phi)=(u,t)\theta$. Thus $u'=u$ and we deduce that for any
$w\in S,t\in T$ we have $(w,t)\theta=(w,t\phi)$. 

\bigskip In each case,  for any $u,u'\in S$ and $t\in T$, by applying $\theta$ to  the product $(u,t)(u',1)$ we immediately see that
$t\,\kappa\, t\phi$. Moreover, as $(s,t_i)\theta=(s,u_i)$ we have
$t_i\phi=u_i$ for $1\leq i\leq n$. Thus $T$ is \ale \, over $T/\kappa$.
\end{proof}

\begin{open} Can we weaken the conditions on $S$ in the converse to  Proposition \ref{semidirect cat}? 
\end{open} 

\begin{example} Let $T$ be a semigroup acting trivially on a finite semigroup $S$, so that $t\cdot s = s$ for each $s\in S, t\in T$. It follows that $\kappa$ is the universal relation. Hence $S\rtimes T$ is $\aleph_0$-categorical if $T$ is $\aleph_0$-categorical over $\{T\}$, which is clearly equivalent to $T$ being $\aleph_0$-categorical. Note that $S\rtimes T$ is simply the direct product of $S$ and $T$, and so we recover Grzegorczyk's result \cite{Grzeg}. 
\end{example} 

\begin{example} Let $L=\{x_1,\dots,x_r\}$ be a finite left zero band and $S=\bigsqcup_{i\in I}^0  S_i$ an $\aleph_0$-categorical 0-direct union of 0-directly indecomposable $S_i$. Then as $L^0=L\cup \{0\}$ is  0-directly indecomposable and $\aleph_0$-categorical, it follow from Proposition \ref{cat 0-direct} that $S'=S\bigsqcup^0 L^0$ is also $\aleph_0$-categorical. Let $S'$ act on its ideal $L^0$ by left multiplication, so that $t\cdot s = ts$ for each $t\in S'$ and $s\in L^0$. Then this is an action by endomorphisms as 
\[ t\cdot (s_1s_2) = ts_1s_2 = 
 \left\{
\begin{array}{ll}
t & \text{if       } t\in L^0\\
0 & \text{else }  
\end{array}
\right. 
= (ts_1)(ts_2).
\] 
We aim to show that $L^0 \rtimes S'$ is $\aleph_0$-categorical. Notice that $t \, \kappa \, t'$ if and only if $sts'=st's'$ for all $s,s'\in L^0$. However, 
\[ sts' = 
 \left\{
\begin{array}{ll}
s & \text{if       } t\in L,\\
0 & \text{else }  
\end{array}
\right. 
\]  
and it follows that the $\kappa$-classes are $L$ and $S'\setminus L$. Since any automorphism of $S'$ which fixes $L$ setwise clearly fixes $S'\setminus L$ setwise,  we have  that  $S'$ is $\aleph_0$-categorical over $S'/\kappa$ if and only if $S'$ is $\aleph_0$-categorical over $\{L\}$. From Lemma \ref{cat over finite}, we have that $S'$ is $\aleph_0$-categorical over  $(x_1,\dots,x_r)$,  hence over $\{L\}$,  and so  over 
$S'/\kappa$.  Hence $L^0 \rtimes S'$ is $\aleph_0$-categorical by Proposition \ref{semidirect cat}. 
\end{example} 

Our final example comes from studying the semidirect product of a group and a semilattice. Such semigroups are examples of $E$-unitary inverse semigroups, a class that plays a central role in the study of inverse semigroups. A semigroup $S$ is \textit{$E$-unitary} if whenever $e,es\in E(S)$ then $s\in E(S)$. In the case of inverse semigroups, this condition is equivalent to $\mathcal{R}\cap\sigma=\iota$ (or, indeed, to
 $\mathcal{L}\cap\sigma=\iota$), a condition often referred to as that of being {\em proper}.  McAlister \cite{McAlister1, McAlister2} showed that every inverse semigroup has an $E$-unitary cover (a pre-image via an idempotent separating morphism) and  characterised the structure of $E$-unitary semigroups via what are known as
 $\mathcal{P}$-semigroups. The construction of a $\mathcal{P}$-semigroup is very close to that of a 
 semidirect product of a semilattice by a group, and certainly embeds into such \cite{ocarroll}. However, we stress that not every $\mathcal{P}$-semigroup is a semidirect product of a semilattice by a group.


 Let $\mathcal{X}$ be a partially ordered set with order $\leq$, and let $\mathcal{Y}$ be an order ideal of ${\X}$ which forms a semilattice under $\leq$. Let $G$ be a group which acts on $\mathcal{X}$ by order automorphisms, and suppose in addition that
\begin{enumerate}
\item $G\mathcal{Y}=\mathcal{X}$ and 
\item $g\mathcal{Y}\cap \mathcal{Y}\neq \emptyset$ for all $g\in G, Y\in \mathcal{Y}$.
 \end{enumerate}   Then the triple $(G,\mathcal{X},\mathcal{Y})$ is called a 
 {\em McAlister triple}. We may then take 
\[ \mathcal{P}=\mathcal{P}(G,{\X},{\Y}) = \{(A,g)\in {\Y}\times G: g^{-1}A\in \Y\}
\] 
and define an operation on $\mathcal{P}$ by the rule 
\[(A,g)(B,h)=(A\wedge gB,gh).\]

\begin{theorem}\cite{McAlister2} Let $(G,\mathcal{X},\mathcal{Y})$ be a  
 {\em McAlister triple}. Then 
$ \mathcal{P}=\mathcal{P}(G,{\X},{\Y})$ is an $E$-unitary inverse semigroup
with semilattice of idempotents $E(\mathcal{P})=\mathcal{Y}\times \{ 1\}$ isomorphic to $\Y$. For any $(A,g),(B,h)\in \mathcal{P}$ we have
\[(A,g)\,\ar\, (B,h)\Leftrightarrow A=B\mbox{ and }
(A,g)\,\el\, (B,h)\Leftrightarrow g^{-1}A=h^{-1}B.\]Moreover, any $E$-unitary inverse semigroup $S$ is isomorphic to some
$\mathcal{P}(G,\mathcal{X},\mathcal{Y})$  where $G=S/\sigma$ and
$\mathcal{Y}=E(S)$.
\end{theorem}

The semigroup $\mathcal{M}(G,\X,\Y)$ is often referred to as a {\em $\mathcal{P}$-semigroup}. Notice that if $\X=\Y$ then $\mathcal{M}(G,\X,\Y)=\Y\rtimes G$.

For any $\mathcal{P}$-semigroup $\mathcal{P}=\mathcal{P}(G,\X,\Y)$ we have $\mathcal{P}/\sigma\cong G$, and so the  $\aleph_0$-categoricity of $\mathcal{P}$ passes to $G$ by Proposition \ref{factor}. Our aim is therefore to consider when the converse holds, or rather,  what conditions on $G$ force $\mathcal{P}$ to be $\aleph_0$-categorical?
 We require  McAlister's \cite{McAlister2}  description of morphisms between $\mathcal{P}$-semigroups, which simplifies to automorphisms as follows. 

\begin{theorem}\label{theorem:paut} Let $\mathcal{P}=\mathcal{M}(G,{\X},{\Y})$ be a $\mathcal{P}$-semigroup. Let $\theta\in \Aut (G)$ and $\psi:\X\rightarrow \X$ an order-automorphism such that $\psi|_{\mathcal{Y}}\in \Aut (\Y)$. Suppose also that, for all $g\in G$ and $A\in \X$, 
\[ (gA)\psi=(g\theta)(A\psi). 
\] 
Then the map $\phi:\mathcal{P}\rightarrow \mathcal{P}$ given by $(A,g)\phi=(A\psi,g\theta)$ is an automorphism, denoted $\phi=(\psi;\theta)$. Conversely, every automorphism of $\mathcal{P}$ is of this type. 
\end{theorem}

Notice that if $\psi$ is the identity, then the automorphism $\theta$ of $G$ is required to preserve the group action, that is, 
\begin{equation} \label{condition1} gA=(g\theta)A \quad  (\forall g\in G)(\forall A\in  \mathcal{X}).
\end{equation} 

 On the other hand, if $\theta$ is the identity, then we require that $\psi$ is a 
 $G$-act morphism, in addition to being an order-automorphism, that is,
\begin{equation} \label{condition2}(gA)\psi=g(A\psi).\end{equation}

We first consider the case of  finite $\Y$. For each $A\in \Y$
let $T_A\subseteq G$ be defined by
\[ T_A=\{ g\in G: g^{-1}A\in \Y\},\] 
noting that $G=\bigcup_{A\in \mathcal{Y}} T_A$ since $g\Y \cap \Y \neq \emptyset$ for each $g\in G$. We also define a relation $\sim_A$  on $T_A$  by:
\[g\sim_A h\Leftrightarrow g^{-1}U=h^{-1}U\mbox{ for all }U\leq A.\]
Clearly each $\sim_A$ is an equivalence. Let $\mathcal{A}=\bigcup_{A\in \mathcal{Y}}T_A/\sim_A$.  

The first part of the following lemma is immediate from \cite[p.217]{Howie94} and the second part is an easy consequence of the first. 
\begin{lemma} \label{lem:mu} Let $\mathcal{P}=\mathcal{P}(G,{\X},{\Y})$ be a $\mathcal{P}$-semigroup. Then 
\[(A,g)\,\mu\, (B,h)\Leftrightarrow A=B\mbox{ and }
g\sim_A h.\] 
Thus $\mu$ has finitely many classes if and only if each $\sim_A$ has  finitely many classes and $\Y$ is finite. 
If $\X$ is finite, then $\mathcal{P}/\mu$ is finite.
\end{lemma}

\begin{proposition}\label{prop:yfinite} Let $\mathcal{P}=\mathcal{P}(G,{\X},{\Y})$ be a $\mathcal{P}$-semigroup such that $\mathcal{P}/\mu$ is finite. Then $\mathcal{P}$ is $\aleph_0$-categorical if and only if $G$ is $\aleph_0$-categorical over $\mathcal{A}$. 
\end{proposition}
\begin{proof} Let $\Y=\{ Y_1,\hdots ,Y_k\}$ and for each $Y_i\in \Y$ pick a set of representatives
\[(Y_i,g_1^i),\hdots, (Y_i,g_{n(i)}^i)\]
of the $\mu$ classes sitting in the $\ar$-class of $(Y_i,1)$, so that 
$g_1^i,\hdots, g_{n(i)}^i$
are representatives of the $\sim_{Y_i}$-classes.

$(\Rightarrow)$ Let $\underline{g}=(a_1,\dots,a_n)$ and $\underline{h}=(b_1,\dots,b_n)$ be a pair of $n$-tuples of $G$. Then for any $j$ with $ 1\leq j\leq n$, as $a_j^{-1} \mathcal{Y} \cap \mathcal{Y} \neq \emptyset$ there exists  $A_j\in \mathcal{Y}$ such that $(A_j,a_j)\in \mathcal{P}$; similarly form $(B_j,b_j)\in \mathcal{P}$. 
 Impose the condition  that 
\[ \big((Y_1,g^1_1),\hdots, (Y_1,g^1_{n(1)}), \hdots,
(Y_k,g^k_1),\hdots, (Y_k,g^k_{n(k)}),(A_1,a_1),\hdots \dots,(A_n,a_n)\big) \]
and 
\[ \big((Y_1,g^1_1),\hdots, (Y_1,g^1_{n(1)}), \hdots,
(Y_k,g^k_1),\hdots, (Y_k,g^k_{n(k)} ),(B_1,b_1),\dots,(B_n,b_n)\big) \]
are related by $\, \sim_{S,t+n} \, $  where $t=|\mathcal{P}/\mu|$, 
via  $\phi=(\psi;\theta)\in \text{Aut}(\mathcal{P})$, say. Clearly each $\mu$-class 
is fixed by $\phi$ and in particular each $A\in\Y$  is fixed by $\psi$. Suppose that
$g\in T_A$. 
Then 
\[ g^{-1}A=(g^{-1}A)\psi =(g^{-1}\theta) (A\psi)=(g^{-1}\theta) A=(g\theta)^{-1}A\]
so that $g\theta\in T_A$. Moreover, for any $U\leq A$ we have $g^{-1}U\in \Y$ and 
\[ g^{-1}U=(g^{-1}U)\psi=(g\theta)^{-1}(U\psi)=(g\theta)^{-1}U,\]
so that $g\sim_A g\theta$. Thus $\theta\in\Aut G$ and $\theta$ preserves $\mathcal{A}$. Clearly 
$a_i\theta=b_i$ for $1\leq i\leq n$. Hence $G$ is $\aleph_0$-categorical over $\mathcal{A}$.
 
$(\Leftarrow)$ Since  $\mathcal{P}/\mu$ and hence $\mathcal{Y}$ are  finite and 
$G$ is $\aleph_0$-categorical over $\mathcal{A}$, from any infinite list of $n$-tuples of elements of
$\mathcal{P}$ we may pick out a pair  $\underline{a}=((A_1,g_1),\dots,(A_n,g_n))$ and $\underline{b}=((A_1,h_1),\dots,(A_n,h_n))$  such that $(g_1,\dots,g_n) \, \sim_{G,\mathcal{A},n} \, (h_1,\dots,h_n)$, via some $\theta\in \text{Aut}(G;\mathcal{A})$. 
Define $\phi:\mathcal{P}\rightarrow \mathcal{P}$ by $(A,g)\phi=(A,g\theta)$. 
Let $g\in G$ and $A,B\in \Y$ with $g^{-1}A\in \Y$. Since $g\in T_A $ and
$A\wedge gB$ exists and is less than $A$, we have 
$g^{-1}(A\wedge  gB)=(g\theta)^{-1}(A \wedge gB)$. Also,
$A\wedge gB=gg^{-1}(A\wedge gB)=g(g^{-1}A\wedge B)=g((g\theta)^{-1}A\wedge B)$. It follows that  $A\wedge gB=A\wedge (g\theta)B$ so that, consequently, $\phi$ is an isomorphism. 
Hence $\mathcal{P}$ is  $\aleph_0$-categorical. 
\end{proof}

Given a $\mathcal{P}$-semigroup $\mathcal{P}=\mathcal{P}(G,{\X},{\Y})$, we define a relation $\nu$ on $G$ by 
\begin{equation} \label{action rel} g \, \nu \, h \Leftrightarrow gA = hA \,\,  (\forall A\in \X). 
\end{equation}
Then $\nu$ is clearly an equivalence relation, and if $g \, \nu \, h$ then $g \, \sim_A \, h$ for any $A\in \Y$, with $g\in T_A$ if and only if $h\in T_A$. 
Hence if $g\in G$ then $g\nu$ is a subset of an element of $T_A/ \sim_A$ for some $A\in \Y$. It follows that if $G$ is $\aleph_0$-categorical over $G/\nu$, then $G$ is $\aleph_0$-categorical over $\mathcal{A}$. 
 If  $\mathcal{X}$ is finite, then  there are only finitely many $\nu$-classes. Lemma~\ref{lem:mu} and Proposition~\ref{prop:yfinite} thus prove the reverse  direction to the following result: 
 
\begin{proposition}\label{prop:yfinite} Let $\mathcal{P}=\mathcal{P}(G,{\X},{\Y})$ be a $\mathcal{P}$-semigroup such that $\mathcal{X}$ is finite. Then $\mathcal{P}$ is $\aleph_0$-categorical if and only if $G$ is $\aleph_0$-categorical over $G/\nu$. 
\end{proposition} 

\begin{proof} Suppose $\mathcal{P}$ is $\aleph_0$-categorical and  $\X=\{X_1,\dots,X_r\}$ is finite. For each $i\in\{ 1,\hdots, r\}$ choose and fix  $g_i\in G$ and $Y_i\in \Y$ such that $X_i=g_iY_i$. Let $\underline{g}=(a_1,\dots,a_n)$ and $\underline{h}=(b_1,\dots,b_n)$ be a pair of $n$-tuples of $G$. For each $1\leq j \leq n$ and $1\leq i \leq r$, form $(A_j,a_j),(B_j,b_j),(C_i,g_i)\in \mathcal{P}$ for some  $A_j,B_j,C_i\in \Y$. Impose the conditions that 
\[ \big( (A_1,a_1),\dots, (A_n,a_n),(C_1,g_1),\dots,(C_r,g_r),(Y_1,1),\dots,(Y_r,1)\big)
\] 
and
\[ \big( (B_1,b_1),\dots, (B_n,b_n),(C_1,g_1),\dots,(C_r,g_r),(Y_1,1),\dots,(Y_r,1)\big) 
\] 
are related by $\sim_{\mathcal{P},n+2r}$, via $\phi=(\psi;\theta)\in \text{Aut}(\mathcal{P})$. 
Then 
\[ X_i\psi=(g_iY_i)\psi=(g_i\theta) (Y_i\psi) = g_i Y_i=X_i
\] 
for each $i$, and so $\psi=I_{\X}$. For any $g\in G$ and $X_i\in \X$ we then have 
\[ gX_i = (gX_i)\psi = (g\theta)(X_i\psi) = (g\theta) X_i,\] 
so that $g \, \nu \, g\theta$. Thus $\theta\in \text{Aut}(G)$ preserves $G/\nu$, and is such that $a_j\theta=b_j$ for $1\leq j \leq n$. Hence $G$ is $\aleph_0$-categorical over $G/\nu$.  

\end{proof}

In the corresponding case where $G$ is finite we consider the augmented
$G$-act $(\X,\Y)$ with universe $\X$ and signature
$(\leq, \Y, \lambda_g (g\in G))$ where $\Y$ is regarded as a unary relation and
$\lambda_g$ is the action of $g$ on $\X$. An automorphism $\phi$ of $(\X,\Y)$ must therefore be a $G$-act isomorphism (that is, $ (gX)\phi=g(X\phi)$ for all $g\in G, X\in \X$), in addition to being an order automorphism of $\X$ fixing $\Y$ setwise.

\begin{proposition} Let $\mathcal{P}=\mathcal{P}(G,{\X},{\Y})$ be a $\mathcal{P}$-semigroup such that $G$ is finite. Then $\mathcal{P}$ is $\aleph_0$-categorical if and only if
$(\X,\Y)$ is   $\aleph_0$-categorical. 
\end{proposition}
\begin{proof} Let $G=\{ g_1,\hdots, g_m\}$ and  pick $D_j\in \mathcal{Y}$ with
$(D_j,g_j)\in \mathcal{P}$, for $1\leq k\leq m$. 

Suppose first that $\mathcal{P}$ is
$\aleph_0$-categorical. For any $n\in \N$ and infinite sequence
of $n$-tuples of $\X$, we can find $h_1,\hdots, h_n\in G$ and  a subsequence in which every $n$-tuple can be written as
$(X_1,\hdots,X_n)$,  where $X_i\in h_i\Y$.  From the $\aleph_0$-categoricity of $\mathcal{P}$ we can find distinct elements 
$(h_1Y_1,\hdots,h_nY_n)$ and $(h_1Z_1,\hdots,h_nZ_n)$ of our sequence (where $Y_i,Z_i\in \Y$ for
$1\leq i\leq n$) such that
\[\big((D_1,g_1),\hdots ,(D_m,g_m),(Y_1,1),\hdots, (Y_n,1)\big)
\sim_{\mathcal{P},m+n} \big((D_1,g_1),\hdots ,(D_m,g_m),(Z_1,1),\hdots, (Z_n,1)\big)\]
via $\phi=(\psi;\theta)$, where clearly $\theta=I_G$. Then for any $g\in G, X\in \X$ we have
\[(gA)\psi=(g\theta)( A\psi)=g(A\psi)\]
so that $\psi$ is a $G$-act isomorphism, in addition to possessing the properties that  $\psi\in \Aut \X$ and $\psi|_{\mathcal{Y}}\in \Aut\Y$. Thus $\psi$ is an automorphism of the augmented $G$-Act $(\X,\Y)$. 
Moreover, we have
\[(h_iY_i)\psi=h_i (Y_i\psi)=h_iZ_i\]
so that 
\[(h_1Y_1, \hdots, h_nY_n)\sim_{(\mathcal{X},\mathcal{Y}), n} (h_1Z_1, \hdots, h_nZ_n)\]
as required. Thus $(\mathcal{X},\mathcal{Y})$ is $\aleph_0$-categorical. 

Conversely, suppose that $(\mathcal{X},\mathcal{Y})$ is $\aleph_0$-categorical and we have an infinite sequence of $n$-tuples of $\mathcal{P}$. Since $G$ is finite and 
 $(\mathcal{X},\mathcal{Y})$ is $\aleph_0$-categorical we may find a distinct pair
 $\big((A_1,h_1),\hdots , (A_n,h_n)\big)$ and 
  $\big((B_1,h_1),\hdots , (B_n,h_n)\big)$ such that
  \[(A_1,\hdots, A_n)\sim_{(\mathcal{X},\mathcal{Y}),n} 
  (B_1,\hdots, B_n)\]
  via $\psi$. As $\psi$ is a $G$-act morphism, it is immediate that
  $(\psi;I_G)$ is in $\Aut(\mathcal{P})$ and moreover,
  \[\big((A_1,h_1),\hdots , (A_n,h_n)\big)\sim_{\mathcal{P},n}\big((B_1,h_1),\hdots , (B_n,h_n)\big)\]
  via $(\psi;I_G)$. Thus $\mathcal{P}$ is $\aleph_0$-categorical as required.
\end{proof} 

To deal with the case of a $\mathcal{P}$-semigroup where both the  semilattice 
$\Y$ and group $G$ are infinite,  we require a little more  sophistication. In the sequel to this article we obtain classes of $\aleph_0$-categorical $E$-unitary semigroups with infinite semilattice of idempotents, by restricting our attention to 
those with central idempotents.

\end{document}